\documentclass[11pt,oneside,reqno]{amsart}
\usepackage{mathpazo} 
\normalfont
\usepackage[T1]{fontenc}

\usepackage[utf8]{inputenc}
\usepackage{amsmath}
\usepackage{amsfonts}
\usepackage{amssymb}
\usepackage{amsthm}
\usepackage{comment}
\usepackage{tikz}
\usepackage{tikz-cd}
\usepackage[all]{xy}
\usepackage[margin=1.2in]{geometry}
\usepackage{amscd}
\usepackage[shortlabels]{enumitem}
\usepackage{stmaryrd}

\DeclareMathOperator{\Spec}{Spec}

\newcommand\Proj{\mathbb{P}_{\mathbb{B}}^n}

\DeclareRobustCommand{\coprod}{\mathop{\text{\fakecoprod}}}
\newcommand{\fakecoprod}{%
	\sbox0{$\prod$}%
	\smash{\raisebox{\dimexpr.9625\depth-\dp0}{\scalebox{1}[-1]{$\prod$}}}%
	\vphantom{$\prod$}%
}

\newcommand{\angles}[1]{\left\langle #1 \right\rangle}

\input xy
\xyoption{all}
\usepackage[colorinlistoftodos]{todonotes}

\definecolor{darkred}{rgb}{0.7,0,0} 

\theoremstyle{definition}
\newtheorem{mydef}{\textbf{Definition}}[section]
\newtheorem{example}[mydef]{\textbf{Example}}

\newtheorem{observ}[mydef]{\textbf{Observation}}

\newtheorem{rmk}[mydef]{\textbf{Remark}}

\newtheorem{myeg}[mydef]{Example}

\theoremstyle{plain}
\newtheorem{mythm}[mydef]{\textbf{Theorem}}

\newtheorem*{nothma}{\textbf{Theorem A}}
\newtheorem*{nothmb}{\textbf{Theorem B}}
\newtheorem*{nothmc}{\textbf{Theorem C}}
\newtheorem*{nothmd}{\textbf{Theorem D}}

\newtheorem{lem}[mydef]{\textbf{Lemma}}
\newtheorem{pro}[mydef]{\textbf{Proposition}}

\newtheorem{corollary}[mydef]{Corollary}

\newtheorem{theorem}[mydef]{Theorem}
\newtheorem{conjecture}[mydef]{Conjecture}
\newtheorem{lemma}[mydef]{Lemma}

\newtheorem{proposition}[mydef]{Proposition}
\newtheorem{definition}[mydef]{Definition}

\newcommand{\Bend}{\text{Bend}}
\newcommand{\Supp}{\text{Supp}}

\begin{document}

\title{Prime Ideals and Projective Closure in the Boolean Polynomial Semiring}

\author{Tristan Bishop}
\address{State University of New York at New Paltz, NY, USA}
\email{bishopt1@newpaltz.edu}

\author{Chris Eppolito}
\address{State University of New York at New Paltz, NY, USA}
\email{eppolito.math@gmail.com}

\author{Radford Green}
\address{Johns Hopkins University, MD, USA}
\email{rgreen87@jh.edu}

\author{Jaiung Jun}
\address{State University of New York at New Paltz, NY, USA}
\email{junj@newpaltz.edu}

\author{Alex Norwood}
\address{State University of New York at New Paltz, NY, USA}
\email{norwooda1@newpaltz.edu}

\author{Daisy Ivy Thackrah}
\address{Cardiff University, Cardiff, Wales}
\email{thackrahd@cardiff.ac.uk}

\makeatletter
\@namedef{subjclassname@2020}{%
	\textup{2020} Mathematics Subject Classification}
\makeatother

\thanks{}

\begin{abstract}
In this paper, we investigate two related questions concerning the Boolean polynomial semiring in two variables. First, we study prime ideals in $\mathbb{B}[x,y]$, beginning with a complete classification of those generated by linear polynomials. Second, motivated by the limitations of ideals in capturing geometric phenomena over semirings, we introduce a notion of projective closure via congruences. We prove that this projective closure coincides with the topological closure in the projective space of prime congruences.
\end{abstract}

\maketitle

\section{Introduction}

A semiring is an algebraic structure satisfying the same axioms as a ring, except that additive inverses are not required. Examples range from familiar semirings such as the natural numbers $\mathbb{N}$ and the nonnegative real numbers $\mathbb{R}_{\geq 0}$ to more exotic ones, such as the tropical semifield $\mathbb{T}$, whose underlying set is $\mathbb{R}\cup\{-\infty\}$, with semiring addition given by the maximum operation and semiring multiplication given by ordinary addition of real numbers. Its subsemifield, the Boolean semifield $\mathbb{B}=\{0,-\infty\}$, is another fundamental example. The tropical and Boolean semifields belong to the important class of \emph{idempotent semirings}, characterized by the identity $a+a=a$ for every element $a$. On the other hand, $\mathbb{N}$ and $\mathbb{T}$ share another important property: for a semiring $S$, whenever $a+b=0_S$, one necessarily has $a=b=0_S$.\footnote{This property is commonly referred to as the \emph{zero-sum-free} (or \emph{negative-free}) property.}

Although these examples may appear somewhat unusual at first, semirings arise naturally in many contexts. For instance, every totally ordered abelian group $G$ gives rise to a semifield $\widetilde{G}=G\cup\{-\infty\}$, where addition is given by the maximum operation and multiplication is the group law. The tropical semifield $\mathbb{T}$ (resp.~the Boolean semifield $\mathbb{B}$) is obtained in this way from $\mathbb{R}$ (resp.~the trivial group $\{0\}$). Another important example comes from commutative algebra: if $R$ is a commutative ring, then the collection of ideals of $R$, equipped with ideal addition and ideal multiplication, forms a commutative semiring. Semirings of this form were studied by the fourth named author in \cite{jun2020lattices}, where an idempotent semiring analogue of Hochster's theorem characterizing spectral spaces was established.

Semiring theory has long been studied as a subject in its own right and has found applications in areas such as optimization, automata theory, and theoretical computer science. In recent years, it has also become increasingly important in tropical geometry, where idempotent semirings and semifields provide the natural algebraic framework. In particular, the tropical semifield $\mathbb{T}$ plays a central role in the theory.

Tropical geometry is a branch of algebraic geometry in which algebraic varieties are studied via their tropicalizations, which are polyhedral complexes encoding substantial information about the original varieties. Given an algebraic variety over a valued field, one obtains its tropicalization by applying the valuation in a coordinate-wise way. When the valuation is nontrivial, the tropicalization naturally takes values in the tropical semifield $\mathbb{T}$, whereas the trivial valuation gives rise to combinatorial objects over the Boolean semifield $\mathbb{B}$.\footnote{For readers familiar with matroid theory, a linear subspace tropicalizes to a valuated matroid in the presence of a nontrivial valuation and to an ordinary matroid under the trivial valuation.} Although $\mathbb{T}$ provides the fundamental algebraic framework for tropical geometry, many questions can first be studied over the Boolean semifield $\mathbb{B}$, which is considerably more tractable while still capturing many of the essential combinatorial features of the theory; see \cite{bernd} for an introduction.

Semiring theory not only provides an algebraic foundation for tropical geometry but also opens up several research directions in another emerging area, $\mathbb{F}_1$-geometry. The goal of $\mathbb{F}_1$-geometry is, broadly speaking, twofold: (1) to develop a theory of schemes over more general bases, such as the semiring $\mathbb{N}$\footnote{See, for instance, \cite{toen2009dessous}.}, and (2) to identify combinatorial cores underlying geometric structures, following the original motivation of Tits \cite{tits1956analogues}. From this perspective, it is natural to consider algebraic geometry over bases such as $\mathbb{N}$, $\mathbb{R}_{\geq 0}$, and other semirings.

For example, a scheme $X_{\mathbb{Z}}$ over $\mathbb{Z}$ may admit a model over $\mathbb{N}$, and one can then study its Boolean fiber $X_{\mathbb{B}}$ to extract combinatorial information about $X_{\mathbb{Z}}$. This perspective was outlined in \cite{borger2024facets}; see also Zelich's thesis \cite{zelich2026miracle} for further developments of this viewpoint. Culling's thesis \cite{culling} also provides very interesting perspective in this direction. 

There are also other important applications of semiring theory. For instance, Connes and Consani have explored the role of semirings in their approach to algebraic and arithmetic geometry \cite{con5,con4,connes2017homological}. In addition, Baker--Bowler's theory of matroids with coefficients is closely related to semiring theory.\footnote{In particular, recent approaches using bands and band schemes, as developed by Baker, Jin, and Lorscheid \cite{baker2025new}, are deeply connected with semiring-theoretic methods.} 

In algebraic geometry, ideals and prime ideals play a fundamental role in making a dictionary between algebra and geometry. But, for semirings, ideals and prime ideals may play a less central role as it may not precisely reflect the geometry. In fact, the absence of additive inverses prevents many familiar arguments from working in the same way as in ring theory. For instance, for a semiring homomorphism $\varphi:R \to R/I$ for an ideal $I$, $\ker(\varphi)$ may be strictly bigger than $I$. A more pathological example is that the fiber product of any two irreducible prime spectra over $\mathbb{B}$ is always irreducible; see \cite[Proposition 5.10]{jun2024equivariant}.

One approach to address this issue is to restrict attention to special classes of ideals in the idempotent setting. For example, a $k$-ideal (also called a subtractive ideal) is an ideal $I$ satisfying the condition that whenever $a+b\in I$ and $a\in I$, then $b\in I$. This condition serves as a substitute for subtraction and restores some desirable properties of ideals in the absence of additive inverses. However, the class of $k$-ideals is often too restrictive and does not capture all naturally occurring ideals in algebraic geometry over semirings. Another important approach is the theory of tropical ideals introduced by Maclagan and Rinc\'on \cite{maclagan2018tropical,maclagan2016tropical}. Tropical ideals are defined through combinatorial conditions and capture many essential features of tropical geometry.

A different and more intrinsic approach is to replace ideals by congruences. A congruence on a semiring $S$ is an equivalence relation $\sim$ on $S$ that is compatible with addition and multiplication; namely, if $a\sim b$ and $c\sim d$, then $(a+c)\sim (b+d)$ and $(ac)\sim (bd)$. In the category of rings, ideals and congruences are essentially equivalent, since an ideal determines a congruence via the relation $a\sim b$ if and only if $a-b$ belongs to the ideal. This correspondence breaks down for semirings because additive inverses are not available. Thus, congruences provide a more natural replacement for ideals and allow one to recover geometric structures in a broader setting. 

In tropical geometry, congruences offer a natural framework for extending geometric ideas from rings to semirings. Several aspects of this perspective have been studied by Jo\'o and Mincheva \cite{joo2018prime,joo2026varieties} and Friedenberg and Mincheva \cite{friedenberg2026geometric}. In particular, Tanaka \cite{Tanaka26} developed a geometric theory of congruences that provides the foundation for the viewpoint adopted in this paper.

\subsection{Summary of Our Results}

In this paper, we explore two questions. The first concerns the structure of prime ideals in $\mathbb{B}[x,y]$, the Boolean polynomial semiring in two variables. In \cite{mincheva2025prime}, Mincheva and Sakran classified the prime ideals of the one-variable Boolean polynomial semiring $\mathbb{B}[t]$, building on earlier work of Alarc\'on and Anderson \cite{alarcon1994commutative}. We extend this line of investigation to $\mathbb{B}[x,y]$. A complete classification of all prime ideals in $\mathbb{B}[x,y]$ appears to be a challenging problem. As a first step, we classify the prime ideals generated by linear polynomials and prove the following result.

\begin{nothma}[Section \ref{section: prime ideals}]
The following is the complete list of prime ideals of $\mathbb{B}[x,y]$, generated by linear polynomials.
\[
\angles{x,y}, \angles{x,y,1+x,1+y}, \angles{x,y,1+x}, \angles{x,y,1+y},  \angles{x}, \angles{y}.
\]
\end{nothma}
We conjecture that for an ideal $I\subseteq \mathbb{B}[x_1,\dots,x_n]$ generated by linear polynomials, if $I$ does not contain all of the variables $x_1,\dots,x_n$, then $I$ is prime only when it is generated by a subset of the variables $\{x_1,\dots,x_n\}$.

We next use geometric intuition to study prime ideals of $\mathbb{B}[x,y]$. As in the classical case of polynomial rings over fields, one can set-theoretically decompose the spectrum as
\begin{equation}\label{eq: eq: decomp}
\operatorname{Spec}(\mathbb B[x,y])=D(xy)\cup V(x)\cup V(y),
\end{equation}
where $D(xy)$ denotes the set of prime ideals in $\mathbb{B}[x,y]$ that do not contain $xy$, and $V(x)$ and $V(y)$ denote the sets of prime ideals containing $x$ and $y$, respectively. Since $D(xy)=D(x)\cap D(y)$, this decomposition can be refined as
\begin{equation}\label{eq: eqdecom2}
\operatorname{Spec}(\mathbb B[x,y])
=
D(xy)\cup (V(x)\cap D(y))\cup (V(y)\cap D(x))\cup V(x,y).
\end{equation}

Over a ring, the pieces involving $V(x)$ and $V(y)$ are closely related to spectra of polynomial rings in fewer variables. Indeed, if $P\subseteq R[x,y]$ is a prime ideal with $x\in P$ and
\[
F=f(y)+xg(x,y),
\]
then $F\in P$ if and only if $f(y)\in P$, since one can subtract the term $xg(x,y)$. This allows one to recover information about primes containing $x$ from the specialization obtained by setting $x=0$.

This fails for semirings because additive inverses are not available. In particular, for a prime ideal $P\subseteq \mathbb{B}[x,y]$ containing $x$, the inclusion
\[
f(y)+xg(x,y)\in P
\]
does not necessarily imply that the $x$-free part $f(y)$ belongs to $P$. Thus, unlike in the ring case, prime ideals containing $x$ are not determined simply by deleting the terms divisible by $x$.

This observation suggests that the strata $V(x)$ and $V(y)$ themselves are too coarse to be analyzed solely through lower-dimensional polynomial semirings. Instead, the more natural pieces to study are the strata
\[
V(x)\cap D(y) \qquad \text{and} \qquad V(y)\cap D(x),
\]
where the remaining variable is nonzero. In other words, our approach amounts to studying specific fibers of a map
\[
\mathbb{B}[x,y]\to \mathbb{B}[t].
\]
This analysis of the resulting strata leads to the following result.

\begin{nothmb}[Theorem \ref{thm: boundary-fiber}]
Let $\mathfrak{q}$ be a prime ideal in $\mathbb{B}[y^\pm]$. Then the following subset of $\Spec (\mathbb{B}[x,y])$
\[
\mathcal F_{\mathfrak q}:=\{\mathfrak p\in V(x) \cap D(y) \mid \mathfrak p\cap \mathbb B[y^{\pm1}]=\mathfrak q\}.
\]  
has the maximal and the minimal elements with respect to the set-inclusion. 
\end{nothmb}
In Remark \ref{rmk: subtractive boundary fiber}, we explain how $\mathcal{F}_{\mathfrak q}$ and its minimal element measure the failure of prime ideals being subtractive (Definition \ref{definition: ideals and primes}).

Our second goal is to study natural geometric spaces associated with $\mathbb{B}[x_1,\dots,x_n]$. To be precise, we aim to develop a notion of projective closure for ideals in the Boolean setting. 

Many important geometric properties become visible only in the projective setting and are not captured by affine varieties alone. It is therefore natural to seek a theory of projective closure in the setting of semirings. We may define homogenization and dehomogenization of primes ideals in $\mathbb{B}[x_1,\dots,x_n]$ as in the classical case (Section \ref{section: homogenization}). However, Example \ref{example: homogenization} shows that the homogenization of a prime ideal need not remain prime over a semiring. Thus, the naive extension of the classical construction of projective closure from rings to semirings does not work.

On the other hand, we show that dehomogenization preserves primality: the dehomogenization of a prime ideal is again a prime ideal. This asymmetry between homogenization and dehomogenization highlights a fundamental difference between ring and semiring settings and motivates our approach using congruences to define a suitable notion of projective closure.

\begin{nothmc}[Example \ref{example: homogenization}, Example \ref{example: homogen2}, and Proposition \ref{pro: dehomogenization prime}]
Let $P\subseteq \mathbb{B}[x_0,\dots,x_n]$ be a prime ideal. 
\begin{enumerate}
    \item 
The associated homogeneous ideal $P^h$ does not have to be a prime ideal. Additionally if $A$ is a prime ideal in $\mathbb{B}[x_1,...,x_n]$ then its homogenization $A^{hom} := \langle F^{hom} | F\in A\rangle$ need not be prime.
\item 
If $P$ is a homogeneous prime ideal, then its dehomogenization (by setting $x_i=1$) is a prime ideal.
\end{enumerate}
\end{nothmc}

Theorem C provides further evidence that congruences, rather than ideals, may be better suited for capturing the geometric structure of semirings. Motivated by this observation, we turn our attention from prime ideals to prime congruences and investigate their geometric and topological properties via projective closure. Accordingly, in the remainder of the paper, congruences will serve as the fundamental objects through which we study the geometry of semirings.

To define an object corresponding to prime ideals, we use the notion of prime congruences, first defined by Bertram and Easton \cite{bertram2017tropical}. Let $R$ be an idempotent semiring, for a congruence $C \subseteq R \times R$, one first defines the \emph{twisted product}: for $x=(x_1,x_2),~~ y=(y_1,y_2) \in C$
\[
(x_1,x_2)\cdot_{\text{tw}} (y_1,y_2) = (x_1y_1+x_2y_2,x_1y_2+x_2y_1). 
\]
$C$ is said to be \emph{prime} if $x\cdot_{\text{tw}}y \in C$ implies either $x \in C$ or $y \in C$. 

The literature contains several notions of prime congruences. Our definition is different from some of these. For instance, Lescot \cite{les3} defines a congruence $C$ to be prime if
\[
(x_1,x_2)\cdot(y_1,y_2):=(x_1y_1,x_2y_2)\in C
\]
implies that either $(x_1,x_2)\in C$ or $(y_1,y_2)\in C$.

In the second part of the paper, we study the geometry of prime congruences, with a particular focus on projective closures of an affine tropical variety over $\mathbb{B}$. Let
\[
\mathbb{A}_{\mathbb{B}}^n:=\operatorname{SpecCong}(A)
\]
be the set of prime congruences on $A$, where $A=\mathbb{B}[t_1,\dots,t_n]$. For a congruence $C$ on $A$, let $V(C)$ denote the set of prime congruences containing $C$. This imposes the Zariski topology on $\operatorname{SpecCong}(A)$ (Lemma \ref{lemma: zariski}).

Our construction combines two fundamental ingredients. The first is the homogenization of prime congruences introduced by Maclagan and Rinc\'on \cite{maclagan2016tropical}. The second is Tanaka's localization theory for congruences \cite{Tanaka26}, which we review in Section \ref{section: preliminaries}. Let $S=\mathbb{B}[x_0,\dots,x_n]$. By passing to the degree-zero part $S_{(x_i)}$ of the localization $S_{x_i}$, we obtain natural identifications
\[
\operatorname{SpecCong}(A)\xrightarrow{\sim} U_i:=\operatorname{SpecCong}(S_{(x_i)}),
\]
where the $U_i$ are the standard affine open subsets of the projective congruence spectrum.

These affine charts allow us to construct the projective space $\Proj$ over $\mathbb{B}$ and to view $V(C)$ as a subspace of $\Proj$. We then define the projective closure $V_+\!\left(C_P^{\mathrm{proj}}\right)$ of an ideal $P\subseteq A$ in $\Proj$ in a chartwise way and prove the following result.

\begin{nothmd}
[Theorem \ref{thm:projective-congruence-closure}]
Let $A=\mathbb{B}[t_1,\dots,t_n]$, $S=\mathbb{B}[x_0,\dots,x_n]$, and $C$ be a congruence on $A$. Let
$C^{\mathrm{proj}}:=(C^h:x_0^\infty)$, where $C^h$ is the homogenization of $C$ and $(C^h:x_0^\infty)$ is the $x_0$-saturation of $C^h$. Then, one has
\[
V_+(C^{\mathrm{proj}})
=
\overline{V(C)}^{\,\Proj},
\]
where $\overline{V(C)}^{\,\Proj}$ denotes the topological
closure of $V(C)$ in $\Proj$. In particular, for every ideal $P\subseteq A$, one has
\[
\operatorname{Proj}(P) := V_+\!\left(C_P^{\mathrm{proj}}\right)
=\overline{V(\operatorname{Bend}(P))}^{\,\Proj}.
\]
\end{nothmd}

As a consequence, we discuss the irreducibility of $\text{Proj}(P)$ in Corollary \ref{cor:projective-irreducibility}.

\bigskip

\textbf{Acknowledgment} J.~Jun was partially supported by the NSF LEAPS-MPS grant (DMS-2532394) and the AMS-Simons Research Enhancement Grant for Primarily Undergraduate Institution (PUI) Faculty during the preparation of this paper. T.~Bishop, R.~Green, A.~Norwood, and D.~Thackrah were also partially supported by NSF LEAPS-MPS grant (DMS-2532394). The authors thank Kalina Mincheva and Jeffrey Tolliver for many helpful conversations and valuable suggestions. The authors also thank Naufil Sakraan and Jeffrey Tolliver for their helpful suggestions and for pointing out several mistakes in an earlier draft.

\bigskip

\section{Preliminaries}\label{section: preliminaries}

In this section, we will review several definitions and propositions which will be used in later sections. All semirings and monoids are assumed to be commutative. 

\subsection{Prime ideals of semirings} 

\begin{mydef}
A \emph{semiring} is an nonempty set $R$ with two binary operations (addition and multiplication) such that $(R,+,0)$ and $(R,\cdot,1)$ are monoids and $0$ is an absorbing element, i.e., $a\cdot 0=0$ for all $a \in R$. Two binary operations are required to be compatible:
\[
a\cdot (b+c) = a\cdot b + a\cdot c. 
\]
A semiring is said to be a \emph{semifield} if $(R\backslash\{0\},\cdot,1)$ is a group. 
\end{mydef}

In this paper, we will only consider \emph{additively idempotent} semirings, i.e., semirings $R$ such that $a+a=a$ for all $a \in R$. The following are two main examples. 

\begin{myeg}[Tropical semifield]
Let $\mathbb{T}=\mathbb{R}\cup \{-\infty\}$. Addition $\oplus$ of $\mathbb{T}$ is to take the maximum of two numbers, i.e.
\[
a\oplus b := \max\{a,b\},
\]
where $-\infty$ is the smallest element. Multiplication $\odot$ is usual addition of real numbers, i.e., 
\[
a\odot b := a+b,\quad a\odot (-\infty) = -\infty.
\]
In the following, we will simply use $+$ and $\cdot$ instead of $\oplus$ and $\odot$.
\end{myeg}

\begin{myeg}[Boolean semifield]
Let $\mathbb{B}=\{0,1\}$. Multiplication of $\mathbb{B}$ is same as $\mathbb{Z}_2$. Addition is given by
\[
1+0 =1, \quad 0+0=0, \quad 1+1=1. 
\]
One may also view $\mathbb{B}$ as a subsemifield of $\mathbb{T}$ by writing $\mathbb{B}=\{-\infty, 0\}$. 
\end{myeg}

\begin{myeg}
Let $\mathbb{B}[x_1,\dots,x_n]$ be the set of formal polynomials with coefficients in $\mathbb{B}$. With formal polynomial addition and multiplication, $\mathbb{B}[x_1,\dots,x_n]$ becomes an idempotent semiring. For instance, we have
\[
(x^2+x+1) + (x^3+x^2+1)=x^3+x^2+x+1
\]
and
\[
(x+1)^3= x^3+x^2+x+1. 
\]
\end{myeg}

\begin{rmk}
We emphasize that we view $\mathbb{B}[x_1,\dots,x_n]$ as formal polynomials rather than polynomial functions. For instance, although $x+1$ and $x^2+1$ are the same functions on $\mathbb{B}$, we view them as different polynomials.  
\end{rmk}

\begin{mydef}\label{definition: ideals and primes}
Let $R$ be a semiring. 
\begin{enumerate}
    \item 
An \emph{ideal} of $R$ is a subset $I\subseteq R$ such that $ra+b \in I$ for all $a,b \in I$ and $r \in R$.
\item 
An ideal $I$ of $R$ is \emph{prime} if $ab \in I$ implies either $a \in I$ or $b \in I$. 
\item 
An ideal $I$ of $R$ is said to be \emph{subtractive} if $a+b \in I$ and $a \in I$ implies that $b \in I$.
\end{enumerate} 
\end{mydef}

As in the ring case, for a subset $A \subseteq R$ of a semiring $R$, we let $\angles{A}$ be the smallest ideal of $R$ which includes $A$. We say $\angles{A}$ is the ideal of $R$ generated by $A$. One can easily check that
\[
\angles{A}=\left\{\sum_{i=1}^n r_ia_i \mid n \in \mathbb{N},~~r_i \in R,~~a_i\in A\right\}.
\]

 \begin{mydef}
For $f \in \mathbb{B}[x_1,\ldots,x_n]$, we let $\text{Supp}(f)$, called the \emph{support} of $f$, be the set of exponent vectors of monomials of $f$ with nonzero coefficients. For example, if $f=x_1^2x_2 + x_1^3x_2^4+1$ then $\text{Supp}(f) = \{(2,1), (3,4),(0,0)\}.$
\end{mydef}

\begin{mydef}\label{definition: prime subset}
A subset $A \subseteq \mathbb{N}$ is \emph{prime} if for all $a + b \in A$, one has $a \in A$ or $b \in A$. In general, for an additive monoid $M$, we call a subset $P\subseteq M$ prime if $a+b \in P$ implies either $a \in P$ or $b \in P$.
\end{mydef}

In \cite{alarcon1994commutative}, Alarc\'on and Anderson introduced the notion of prime subsets of $\mathbb{N}$ and used it to partially classify the prime ideals of $\mathbb{B}[t]$. The underlying idea is that multiplication in $\mathbb{B}[t]$ depends heavily on the supports of polynomials, making the study of prime ideals closely related to the combinatorics of prime subsets of $\mathbb{N}$.
 
In \cite{alarcon1994commutative}, to a subset $A\subseteq \mathbb{N}$, Alarc\'on and Anderson first associated two ideals:
\begin{equation}\label{eq: I and J}
I_A:=\angles{t, \{1+t^a \mid a \in A\}}, \quad J_A:=I_A - \{ t^{m}f \; | \; m \ge 0, f \notin I_A \}.
\end{equation}
By definition $I_A$ is an ideal. Alarc\'on and Anderson showed that $J_A$ was an ideal by proving the following equality:
\begin{equation}\label{eq: J_A}
J_A=\angles{\{1+t^a+t^m \mid a \in A, m\geq 0\}}.  
\end{equation}

With this terminology, Alarc\'on and Anderson proved the following. 

\begin{theorem}[{\cite[Theorem 11]{alarcon1994commutative}}]\label{theorem: alarcon and anderson}
Let $A\subseteq \mathbb{N}$. Then the following hold. 
\begin{enumerate}
  \item 
  $I_A$ a prime ideal if and only if $A$ is a prime subset.
  \item 
  $J_A$ is a prime ideal if and only if $A$ is a prime subset.  
\end{enumerate}
\end{theorem}

As not all prime ideals arise as $I_A$ or $J_A$, Theorem \ref{theorem: alarcon and anderson} does not completely classify prime ideals of $\mathbb{B}[t]$. They conjectured that all prime ideals of $\mathbb{B}[t]$ is either $I_A$ or $J_A$ for some prime subset $A\subseteq \mathbb{N}$. 

In \cite{mincheva2025prime}, Mincheva and Sakran disproved this conjecture and completed the classification. We recall some observations from \cite{mincheva2025prime} as  we will apply similar ideas later.

\begin{observ}\cite[Section 2]{mincheva2025prime}\label{observation: basics from Kalina's}
 \begin{enumerate}
  \item 
In $\mathbb{B}[t]$, one has
\begin{equation}
(1+t^2+t^3)(1+t+t^3)=(1+t)^6=(1+t)(1+t^2)(1+t^3).
\end{equation}
In particular, the ideal generated by $(1+t)$ is not prime even though $(1+t)$ is irreducible. In fact, $\mathbb{B}[t]$ is not UFD and $(1+t)$ is an irreducible element, but not a prime element. 
\item 
If $f \in \mathbb{B}[t]$ has degree $N$ and $0 \in \text{Supp}(f)$, then one has
\begin{equation}
f(t)(1+t)^N = (1+t)^{2N}.
\end{equation}
It follows that every prime ideal of $\mathbb{B}[t]$ contains $(1+t)$, except for the prime ideal generated by $t$. 
\end{enumerate}  
\end{observ}

\subsection{Congruences on semirings}\label{subsection: congrouence}

In this subsection, we briefly recall prime congruences and some related definitions. For the interested reader, we refer to \cite{mincheva2016semiring,giansiracusa2016equations,joo2018prime,maclagan2018tropical, maclagan2016tropical} for details. 

\begin{mydef}
Let $R$ be an additively idempotent semiring. 
\begin{enumerate}
  \item 
A \emph{congruence} on $R$ is an equivalence relation $\sim$ on $R$ which is compatible with addition and multiplication, i.e, for any $a,b,c,d \in R$:
\[
 a\sim b,~~ c\sim d \implies (a+c) \sim (b+d) \text{ and } (ac)\sim (bd).
\]
We will interchangeably use $\sim_C$ and $C$ to denote congruences. We often use $\sim$ whenever $C$ is clear from context. 
\item 
For pairs $\alpha=(a,b)$ and $\beta=(c,d)$ in $R\times R$, the \emph{twisted product} is defined as follows:\footnote{See \cite{joo2018prime} and \cite{joo2026varieties} for the motivation and applications.}
\[
\alpha\beta=(ac+bd,ad+bc).
\]
\item 
A \emph{prime congruence} is a proper congruence $P$ on $R$ such that whenever $\alpha\beta\in P$ (twisted product) for $\alpha,\beta\in R\times R$, one has $\alpha\in P$ or $\beta\in P$.
\end{enumerate}
\end{mydef}

\begin{mydef}\cite{maclagan2016tropical}
Suppose that $R=\bigoplus_{d\geq 0}R_d$ is a graded semiring. A congruence $C$ on $R$ is called \emph{homogeneous} if it is generated, as a congruence, by relations $a\sim b$ such that $a$ and $b$ are homogeneous of the same degree. Equivalently, $C$ is generated by degree-preserving relations, so that the quotient $R/C$ inherits a grading from $R$.  
\end{mydef}

For any polynomial $F=\sum a_{\mathbf{u}}x^{\mathbf{u}}$ with coefficients in a semiring,  we define the \emph{homogenization} of $F$ as follows:
\begin{equation}
F^h:=\sum a_{\mathbf{u}}x^\mathbf{u}x_0^{\deg{F}-|\mathbf{u}|},
\end{equation}
where $|\mathbf{u}|=u_1+\dots u_n$ for $\mathbf{u}=(u_1,\dots,u_n) \in \mathbb{N}^n$. 

\begin{mydef}
\cite[Definition~2.8]{maclagan2016tropical}
Let $A=\mathbb B[x_1,\ldots,x_n]$ and
$S=\mathbb B[x_0,\ldots,x_n]$.
Let $F\sim G$ be a relation in $A$. If neither side is zero, let
$D=\max\{\deg(F),\deg(G)\}$ and define its homogenization, which is a
relation in $S$, by
\begin{equation}\label{eq: homoge notation}
(F\sim G)^h
:=
x_0^{D-\deg(F)}F^h
\sim
x_0^{D-\deg(G)}G^h.
\end{equation}
If exactly one side is zero, take $D$ to be the degree of the nonzero
side and homogenize the zero side as $0$; if both sides are zero, the
relation is trivial.
\end{mydef}

\begin{mydef}
For a congruence $C$ on $A=\mathbb{B}[x_1,\dots,x_n]$, define its homogenization to be the congruence on $S=\mathbb B[x_0,\ldots,x_n]$:
\begin{equation}
C^h:=\langle(F\sim G)^h \mid (F\sim G)\in C\rangle\subseteq S\times S.
\end{equation}    
\end{mydef}

\begin{mythm}\cite[Proposition 2.7]{maclagan2016tropical}\label{theorem: congruence homogeneous elts}
 Let $S=\mathbb{B}[x_0,\dots,x_n]$ and $C$ be a homogeneous congruence on $S$. If $(F,G) \in C$, then $(F_d,G_d) \in C$ for all $d \in \mathbb{N}$, where $F_d$ is the degree $d$ part of $F$.
\end{mythm}

In \cite{giansiracusa2016equations}, Jeffrey and Noah Giansiracusa introduced \emph{bend congruences}, providing a scheme-theoretic framework for tropical geometry in which the tropicalization of an algebraic variety is realized as the set of $\mathbb{T}$-points of an associated tropical scheme. Building on this perspective, Maclagan and Rinc\'on \cite{maclagan2018tropical} introduced tropical ideals, giving a purely ideal-theoretic and combinatorial framework for studying bend congruences.

\begin{mydef}\label{definition: bend from ideal}\cite{giansiracusa2016equations}\label{definition bend}
 Let $P\subseteq \mathbb{B}[x_1,\dots,x_n]$ be an arbitrary ideal. The \emph{bend congruence} of $P$ is defined as follows:
\begin{equation}
\Bend(P):=\langle f\sim f_{\widehat{\mathbf u}}\mid f\in P,\ \mathbf u\in\operatorname{Supp}(f)\rangle, 
\end{equation}
where $f_{\widehat{\mathbf u}}$ denotes the polynomial obtained from $f$ by deleting the term $t^{\mathbf u}$ written in the multi-index notation. 
\end{mydef}

\subsection{Localization of semirings and Prime congruence spectra}\label{subsection: localization}

In this section, we briefly review the notion of localization of semirings and how congruences behave under it. We refer the reader to \cite{Tanaka26} for more details.

Let $R$ be a semiring, and let $T\subseteq R$ be a multiplicative subset. Unless otherwise specified, we always assume that $1\in T$ and $0\notin T$.

As in the case for rings, we define an equivalence relation $\sim$ on $R\times T$:
\[
(a,b) \sim (c,d) \iff adt = bct \text{ for some $t \in T$}.
\]
Denote the equivalence class of $(a,b)$ by $\frac{a}{b}$. One can easily see that $(R\times T)/\sim$ is a semiring with 
\[
\frac{a}{b}+\frac{c}{d}:=\frac{ad+bc}{bd}, \quad \frac{a}{b}\frac{c}{d}:=\frac{ac}{bd}.
\]
One further has a canonical homomorphism (localization):
\[
\phi:R \to T^{-1}R, \quad a \mapsto \frac{a}{1}.
\]

Now, one has the following.

\begin{mydef}\cite[Section 2]{Tanaka26}
Let $T\subseteq$ be multiplicative subset of an idempotent semiring $R$ and $C$ be a congruence on $R$. The \emph{localization} $T^{-1}C$ is the congruence on $T^{-1}R$ defined as follows:
\[
T^{-1}C:=\angles{\frac{a}{1} \sim \frac{b}{1}}_{(a,b) \in C}.
\]
If $T=\{1,x,x^2,\dots\}$ for some $x \in R$, we write $R_{x_i}$ for $T^{-1}R$.
\end{mydef}

Suppose that an idempotent semiring $R$ is graded and $x_i$ is an homogeneous element of $R$. Then, $R_{x_i}$ is naturally equipped with grading:
\begin{equation}\label{eq: graded}
\deg(a/x_i^k):=\deg(a) - k\deg(x_i).
\end{equation}
We let $(R_{x_i})_0$ be the degree-zero component of $R_{x_i}$, viewed as a graded semiring with \eqref{eq: graded}.

\begin{mydef}\label{definition: degree-zero}
Let $R$ be a graded idempotent semiring and $x_i$ is an homogeneous element. By the \emph{degree-zero localization}, we mean the following set:
\[
C_{(x_i)}:=C_{x_i}\cap((R_{x_i})_0\times(R_{x_i})_0).
\]
\end{mydef}

Since $C_{x_i}$ is a congruence on $R_{x_i}$, its restriction
\[
C_{(x_i)}
=
C_{x_i}\cap
\bigl((R_{x_i})_0\times(R_{x_i})_0\bigr)
\]
is a congruence on $(R_{x_i})_0$ by the following lemma.

\begin{lemma}
Let $\varphi:R\to R'$ be a homomorphism of semirings and let $C$ be
a congruence on $R'$. Then
\[
\varphi^{-1}(C)
:=
\{(a,b)\in R\times R:(\varphi(a),\varphi(b))\in C\}
\]
is a congruence on $R$.
\end{lemma}
\begin{proof}
It is clear that $\varphi^{-1}(C)$ is an equivalence relation. The fact that it's a congruence relation follows from the fact that $C$ is a congruence relation and $\varphi$ is a homomorphism. 
\end{proof}

Let $R$ be a semiring. By an $R$-algebra, we mean a semiring $A$ equipped with a semiring homomorphism (structure homomorphism) $\iota:R\to A$, satisfying the analogous axioms to those in the classical setting. For instance, $\mathbb{B}[x_1,\dots,x_n]$ is a $\mathbb{B}$-algebra with a natural map $\iota:\mathbb{B}\to \mathbb{B}[x_1,\dots,x_n]$.

Let $A$ be a $\mathbb{T}$-algebra with structure homomorphism
$\iota:\mathbb{T}\rightarrow A$. Let
\[
\Delta_{\mathbb{T}}
:=
\{(a,a)\mid a\in\mathbb{T}\}
\]
denote the diagonal congruence on $\mathbb{T}$. For a congruence $Q$ on $A$, we define the restriction of $Q$ to $\mathbb{T}$ via $\iota$ by
\[
Q|_{\mathbb{T}}
:=
\{(a,b)\in\mathbb{T}\times\mathbb{T}\mid
(\iota(a),\iota(b))\in Q\}.
\]

Following \cite[Definition~2.15]{Tanaka26}, a congruence $Q$ on $A$
is said to \emph{lie over $\mathbb T$} if the composite
\begin{equation}
\mathbb T\xrightarrow{\iota}A\longrightarrow A/Q    
\end{equation}
is injective, or equivalently, if
\[
Q|_{\mathbb T}=\Delta_{\mathbb T}.
\]

A congruence $Q$ on $A$ is called
\emph{geometric} if the induced map
$\mathbb T\longrightarrow A/Q$
is an isomorphism.

\begin{mydef}\label{definition: prime congreunce spectra}
Let $A$ be an idempotent semiring.
\begin{enumerate}
    \item 
The \emph{prime congruence spectrum} of $A$ is defined as follows:
\begin{equation}
\operatorname{SpecCong}(A):=\{Q\subseteq A\times A \mid Q\text{ is a prime congruence on }A\}.  
\end{equation}
\item 
If $C$ is a congruence on $A$, define
\begin{equation}
V(C):=\{Q\in\operatorname{SpecCong}(A) \mid C\subseteq Q\}.    \end{equation}
\item 
If $A=\mathbb{B}[t_1,\dots,t_n]$, we write $\mathbb{A}_{\mathbb{B}}^n$ for $\operatorname{SpecCong}(A)$, called an \emph{affine space} of dimension $n$ over $\mathbb{B}$.
\end{enumerate}
\end{mydef}

The following result is elementary, but we include the proof for completeness and for readers who are less familiar with congruences.

\begin{lem}\label{lemma: zariski}
Let $A$ be an idempotent semiring. The sets $V(C)$ define a topology on $\operatorname{SpecCong}(A)$ by declaring them to be closed.
\end{lem}
\begin{proof}
Let $\Delta$ be the diagonal congruence. Then, since prime congruences are proper, we have
\[
V(\Delta)=\operatorname{SpecCong}(A) \quad \text{and} \quad V(A\times A)=\emptyset.
\]
If $\{C_\lambda\}_{\lambda\in\Lambda}$ is any family of congruences, then 
\begin{equation}
\bigcap_{\lambda\in\Lambda}V(C_\lambda)=V\left(\left\langle\bigcup_{\lambda\in\Lambda}C_\lambda\right\rangle\right).   
\end{equation}
Indeed, for any $Q\in\operatorname{SpecCong}(A)$, one has 
\begin{equation}\label{eq: equi con}
 Q\in\bigcap_{\lambda\in\Lambda}V(C_\lambda)
\Longleftrightarrow
C_\lambda\subseteq Q~\forall\lambda \in \Lambda
\Longleftrightarrow
\bigcup_{\lambda\in\Lambda}C_\lambda\subseteq Q.   
\end{equation}
Since $Q$ is a congruence, \eqref{eq: equi con} is equivalent to 
\[
\left\langle\bigcup_{\lambda\in\Lambda}C_\lambda\right\rangle\subseteq Q,
\]
where $\left\langle\bigcup_\lambda C_\lambda\right\rangle$ denotes the smallest congruence containing all the $C_\lambda$. Hence, we have
\[
\bigcap_{\lambda\in\Lambda}V(C_\lambda) = V\left(\left\langle\bigcup_{\lambda\in\Lambda}C_\lambda\right\rangle\right).
\]

For finite unions, define the product congruence $CD:=\langle \alpha\beta:\alpha\in C,\ \beta\in D\rangle,$ where for pairs $\alpha \beta$ is given by the twisted product. Then, we have
\begin{equation}
   V(C)\cup V(D)=V(CD). 
\end{equation}
The inclusion $V(C)\cup V(D)\subseteq V(CD)$ is immediate. Conversely, if $Q\in V(CD)$ and $C\not\subseteq Q$, choose $\alpha\in C\setminus Q$. For every $\beta\in D$, one has $\alpha\beta\in CD\subseteq Q$. Since $Q$ is prime and $\alpha\notin Q$, it follows that $\beta\in Q$. Hence $D\subseteq Q$, so $Q\in V(D)$. Thus $V(CD)\subseteq V(C)\cup V(D)$. Therefore the sets $V(C)$ satisfy the closed-set axioms.
\end{proof}

In what follows, we consider $\operatorname{SpecCong}(A)$ for an idempotent semiring $A$ as a topological space with topology in Lemma \ref{lemma: zariski}.

\begin{mydef}\cite{Tanaka26}
The \emph{relative prime-congruence spectrum} of $A$ over $\mathbb T$ is
\[
\operatorname{SpecCong}_{\mathbb T}(A)
:=
\left\{
Q\in\operatorname{SpecCong}(A) \mid 
Q|_{\mathbb T}=\Delta_{\mathbb T}
\right\},
\]
equipped with the subspace topology inherited from
$\operatorname{SpecCong}(A)$. A prime congruence of $A$ which are in $\operatorname{SpecCong}_{\mathbb T}(A)$ is said to be \emph{lying over} $\mathbb{T}$.
\end{mydef}

\begin{rmk}
Tanaka does not use the terminology ``relative prime-congruence spectrum.'' Instead, he uses two different notations, $\operatorname{SpecCong}(A)$ and $\operatorname{SpecCong}_{\mathbb T}(A)$, depending on whether $A$ is viewed as a semiring or as a $\mathbb{T}$-algebra, respectively. 

In our setting, this distinction does not arise over $\mathbb{B}$, since every idempotent semiring has a canonical structure of a $\mathbb{B}$-algebra. More precisely, let $Q$ be a proper congruence on a $\mathbb{B}$-algebra $A$. Then the restriction of $Q$ to $\mathbb{B}$ is necessarily the diagonal congruence. Indeed, if it were not, then we would have $0\sim_Q 1$, which would imply that $Q$ is the improper congruence. Therefore, every prime congruence on a $\mathbb{B}$-algebra lies over $\mathbb{B}$.
\end{rmk}

\begin{mythm}\cite{Tanaka26}\label{theorem: Tanaka}
Let $A$ be an idempotent semiring. 
\begin{enumerate}
    \item 
If $C$ is a congruence on $A$, then pullback along the
    quotient map $\pi_C:A\to A/C$ induces a homeomorphism
    \[
    \operatorname{SpecCong}(A/C)
    \xrightarrow{\sim}
    V(C).
    \]
\item 
Let $T\subseteq A$ be a multiplicative subset and
    let $\lambda:A\to T^{-1}A$ be the localization map. Then pullback
    induces a topological embedding
    \[
    \operatorname{SpecCong}(T^{-1}A)
    \longrightarrow
    \operatorname{SpecCong}(A)
    \]
    whose image is $D(T):=
    \left\{
    Q\in\operatorname{SpecCong}(A):
    (t,0)\notin Q\text{ for every }t\in T
    \right\}.$ If $T$ is finitely generated as a monoid, then $D(T)$ is open.
\item 
If $C\cap(T\times\{0\})=\varnothing$, then there is a natural isomorphism
    \[
    T^{-1}A/T^{-1}C
    \cong
    \pi_C(T)^{-1}(A/C),
    \]
    and the induced map on prime-congruence spectra identifies
    \[
    \operatorname{SpecCong}\!\left(
    \pi_C(T)^{-1}(A/C)
    \right)
    \]
    homeomorphically with $V(C)\cap D(T)$.

    \item Let $\lambda:A\to T^{-1}A$ be the localization map and let
    $C$ be a congruence on $A$. Then
    \[
    T^{-1}C
    =
    \left\{
    \left(\frac{a}{t},\frac{b}{t}\right)  \mid
    (a,b)\in C,\ t\in T
    \right\}.
    \]
    Consequently,
    \[
    \frac{a}{s}\sim_{T^{-1}C}\frac{b}{t} \iff \exists~v \in T \textrm{ such that }  (vta,vsb)\in C.
    \]
\end{enumerate}
\end{mythm}

\begin{proof}
The first assertion is \cite[Proposition~2.8]{Tanaka26}, the second is
\cite[Proposition~2.12]{Tanaka26}, and the third is
\cite[Proposition~2.13]{Tanaka26}. The fourth assertion is \cite[Lemma~2.9]{Tanaka26}; we prove the
stated consequence. For the consequence, if
$(vta,vsb)\in C$, then
\[
\frac{a}{s}
=
\frac{vta}{vst}
\sim_{T^{-1}C}
\frac{vsb}{vst}
=
\frac{b}{t}.
\]
Conversely, Tanaka's description gives a relation $(c,d)\in C$ and
$u\in T$ such that
\[
\frac{a}{s}=\frac{c}{u}
\quad \textrm{and}\quad
\frac{b}{t}=\frac{d}{u}.
\]
By the defining equality relation in the localization, there exist
$r_1,r_2\in T$ such that
\[
r_1au=r_1cs
\qquad\text{and}\qquad
r_2bu=r_2dt.
\]
Since $(c,d)\in C$, compatibility of $C$ with multiplication gives
\[
(r_1r_2stc,r_1r_2std)\in C.
\]
Using the two displayed equalities, this becomes
\[
(r_1r_2uta,r_1r_2usb)\in C.
\]
Taking $v=r_1r_2u\in T$ proves the stated consequence.
\end{proof}

The notion of radicals was introduced in \cite[Definition~3.1]{joo2018prime}. Tanaka proved the following.

\begin{proposition}\cite[Lemma~2.6]{Tanaka26}
\label{prop:Tanaka-irreducibility}
For a congruence $C$ on an idempotent semiring $A$, set
\[
\sqrt[\operatorname{pr}]{C}
:=
\bigcap_{\substack{Q\in\operatorname{SpecCong}(A)\\ C\subseteq Q}}Q.
\]
Then, one has $V(C)=V\!\left(\sqrt[\operatorname{pr}]{C}\right)$. If $V(C)\neq\varnothing$, then $V(C)$ is irreducible if and only if
$\sqrt[\operatorname{pr}]{C}$ is a prime congruence. In that case, one has 
\[
V(C)
=
V\!\left(\sqrt[\operatorname{pr}]{C}\right)
=
\overline{\left\{\sqrt[\operatorname{pr}]{C}\right\}}.
\]
\end{proposition}

\section{Prime ideals generated by linear polynomials in $\mathbb{B}[x,y]$}\label{section: prime ideals}

In this section, we classify prime ideals in $\mathbb{B}[x,y]$ that are generated by linear polynomials. In the one-variable case, the classification follows directly from \cite[Theorem 5.1]{mincheva2025prime} or \cite{alarcon1994commutative}: the only prime ideals in $\mathbb{B}[t]$ generated by linear polynomials are
\[
\langle t\rangle,\qquad \langle t,1+t\rangle.
\]
The situation becomes considerably more subtle in $\mathbb{B}[x,y]$, as we will see below. We divide our analysis according to the possible combinations of $x$ and $y$ being contained in a prime ideal.

\begin{pro}\label{pro1}
The prime ideals of $\mathbb{B}[x,y]$ generated by linear polynomials containing both $x$ and $y$ are the following:
\begin{equation}\label{eq: primes}
\angles{x,y},~\angles{x,y,1+x,1+y},~\angles{x,y,1+x},~\angles{x,y,1+y}.
\end{equation}
\end{pro}
\begin{proof}
One can easily see by \cite[Theorem 5.1]{mincheva2025prime} that \eqref{eq: primes} are prime ideals with the prime subsets (by using the same notation as in \cite{mincheva2025prime}) 
\[
\emptyset,\quad \{(1,0),(0,1)\}, \quad \{(1,0)\}, \quad \{(0,1)\},
\]
respectively. Also, these are all of the possible cases which contain both $x$ and $y$ except $I=\angles{x,y,1+x+y}$ - in this case, $(1+x)(1+y)=yx+(1+x+y) \in I$, but $(1+x) \not \in I$ and $(1+y) \not \in I$, showing that $I$ is not prime. 

Note that $x+y$ is always in the ideal as we assume that $x,y$ are in the ideal. This proves the proposition. 
\end{proof}

Our classification relies on a detailed study of membership criteria for prime ideals. As an illustration of this approach, we present the following lemma. In the following lemma, instead of using the vector notation for support, we will use the monomial notation to ease the notation. For instance, for $f=1+y+xy$, instead of $\Supp(f)=\{(0,0),(0,1),(1,1)\}$, we simply write $\Supp(f)=\{1,y,xy\}$. 

\begin{lem}\label{lem: x,1+x,x+y}
Let $I=\angles{x,1+x,x+y}\subseteq \mathbb{B}[x,y]$. Then $f \in I$ if and only if the following conditions hold: 
\begin{enumerate}
    \item 
If $1 \in \Supp(f)$, then $x \in \Supp(f)$. 
    \item 
If $y^j \in \Supp(f)$ for $j\geq 1$, then $xy^j \in \Supp(f)$ or $xy^{j-1} \in \Supp(f)$.  
\end{enumerate}
\end{lem}
\begin{proof}
First, suppose that $f \in I$. Then we can write $f=ax+b(1+x)+c(x+y)$ for some $a,b,c \in \mathbb{B}[x,y]$. So, we can rewrite:
\[
f=(a+b+c)x+b+cy.
\]
If $1 \in \Supp(f)$, then $1 \in \Supp(b)$ and hence $x\in \Supp(f)$. Likewise, if $y^j$ for $j\geq 1$ is in $\Supp(f)$, then either $y^j \in \Supp(b)$ (in which case $xy^j \in \Supp(f)$) or $y^j \in \Supp(cy)$ (in which case $xy^{j-1} \in \Supp(c(x+y))$. 

Conversely, suppose that $f$ satisfies the given condition. If $1 \not\in \Supp(f)$, then we can write $f=g+h$, where $g$ contains all monomials divisible by $x$ and $h=h(y)$, a polynomial in $y$, such that $1 \not \in \Supp(h)$. By the given condition, for each $y^i \in \Supp(h)$ either $xy^i \in \Supp(f)$ or $xy^{i-1} \in \Supp(f)$. We let $\alpha$ be the sum of monomials $y^i$ of $h$ such that $xy^i \in \Supp(f)$, and $\beta$ be the sum of monomials $y^j$ of $h$ such that $xy^{j-1} \in \Supp(f)$. Let $\beta'$ the be sum of monomials of $\beta$ obtained by reducing the degree of the monomials of $\beta$ by one. Then we have $h=\alpha+\beta$. Note that $\Supp(\alpha)$ and $\Supp(\beta)$ do not have to be disjoint, but since $\mathbb{B}[x,y]$ is idempotent the repetition does not change the sum $\alpha+\beta$. Now, we have
\[
f=g+h = g+\alpha +\beta +x\alpha +x\beta'=g+(1+x)\alpha+(x+y)\beta'. 
\]
It follows that $f \in I$. 

Now, suppose that $1 \in \Supp(f)$. Then, from the given condition, we have $x \in \Supp(f)$. Let $f=1+f'$, where $1 \not \in \Supp(f')$. We further write $f'=g'+h'$, where $g'$ contains all monomials divisible by $x$ and $h'$ is not divisible by $x$. We have
\[
f=1+f'= 1+x+f'=(1+x) + g' + h'.
\]
Now, since $h'$ is not divisible by $x$ and $1 \not \in \Supp(h')$, we may write $h'=h'(y)$, a polynomial in $y$ with no constant term. Now, we may apply the same trick as above and write $h'=\alpha+\beta$, based on whether $xy^i \in \Supp(f)$ or $xy^{i-1} \in \Supp(f)$. So, we have
\[
f=1+g'+\alpha+\beta=(1+x) +g' +\alpha+\beta +x\alpha +x\beta' = (1+x) +g' +(1+x)\alpha +(x+y)\beta', 
\]
showing that $f \in I$.
showing that $f \in I$.
\end{proof}

\begin{pro}\label{proposition: x case}
The only prime ideal on $\mathbb{B}[x,y]$, generated by linear polynomials, containing $x$ but not $y$ is $\angles{x}$.
\end{pro}
\begin{proof}
\underline{Case 1: $\angles{x}$}; this is clearly a prime ideal.

\underline{Case 2: $\angles{x, 1+x}$}; this is not a prime ideal. One may notice the following:
\begin{equation}\label{eq: case2}
(1+y+xy)(x+y)=x(1+xy)+(y+y^2)(1+x) \in \angles{x,1+x}.
\end{equation}
But, clearly $1+y+xy$ and $x+y$ are not elements of $\angles{x,1+x}$.

\underline{Case 3: $\angles{x, 1+y}$}; this is not prime by the exact same reason as in \cite{mincheva2025prime}. To be precise, 
\begin{equation}\label{eq: 1+y}
(1+y+y^3)(1+y^2+y^3)=(1+y)(1+y^2)(1+y^3),
\end{equation}
but $(1+y+y^2)$ and $(1+y^2+y^3)$ are not in $\angles{x,1+y}$.

\underline{Case 4: $\angles{x, x+y}$}; this is not a prime ideal. To be precise, we have \begin{equation}
(x+y+y^2)(x+y+1)=(x+y^2)x +(y^2+y+1)(x+y) \in \angles{x,x+y}.
\end{equation}
But neither $(x + y + y^{2})$ nor $(x + y+1)$ is in the ideal.

\underline{Case 5: $\angles{x, 1+x+y}$}; this is not a prime ideal. To be precise, we have
\begin{equation}\label{eq: 1+x+y}
(1+x)(1+y) = yx +(1+x+y) \in \angles{x,1+x+y}
\end{equation}
but $1+x$ and $1+y$ are not in $\angles{x,1+x+y}$.

\underline{Case 6: $\angles{x, 1+x, 1+y}$}; this is not a prime ideal. The same equation as in \eqref{eq: 1+y} can be used to prove this since $1+y+y^3$ and $1+y^2+y^3$ are not in $\angles{x, 1+x, 1+y}$.

\underline{Case 7: $\angles{x, 1+x, x+y}$}; this is not a prime ideal. In fact, it is clear that $y^2+x$ and $y^3+xy^3+y$ are not in $\angles{x,1+x,x+y}$. But, we have
\begin{equation}\label{eq: case 7}
(y^2+x)(y^3+xy^3+y)=(1+x)y^5+(1+x)y^3+x(y^3+xy^3+y) \in \angles{x,1+x,x+y},
\end{equation}
showing that $\angles{x,1+x,x+y}$ is not prime.

\underline{Case 8: $\angles{x, 1+x, 1+x+y}$}; this is not a prime ideal. We may consider the same equation as in \eqref{eq: case2}. Then $x(1+xy)+(y+y^2)(1+x) \in \angles{x,1+x,1+x+y}$, but clearly $1+y+xy$ and $x+y$ are not in $\angles{x,1+x,1+x+y}$.

\underline{Case 9: $\angles{x, 1+y, x+y}$}; this is not a prime ideal. In fact, we can use the same equation \eqref{eq: 1+y} as in Case 3 since $1+y+y^3$ and $1+y^2+y^3$ are not in $\angles{x,1+y,x+y}$.

\underline{Case 10: $\angles{x, 1+y, 1+x+y}$}; this is same as Case 3 since $1+x+y=(1+y)+x$. 

\underline{Case 11: $\angles{x, x+y, 1+x+y}$}; this is not a prime ideal by \eqref{eq: 1+x+y}.

\underline{Case 12: $\angles{x, 1+x, 1+y, x+y}$}; this is not a prime ideal. In fact, note that we can use the same identity \eqref{eq: case 7} in Case 7, i.e, we have
\[
(y^2+x)(y^3+xy^3+y) \in \angles{x,1+x,1+y,x+y},
\]
but it is clear that $y^2+x$ and $y^3+xy^3+y$ are not in $\angles{x,1+x,1+y,x+y}$.

\underline{Case 13: $\angles{x, 1+x, 1+y, 1+x+y}$}; this is same as Case 6 since $1+x+y=(1+x)+(1+y)$.

\underline{Case 14: $\angles{x, 1+x, x+y, 1+x+y}$}; this is same as Case 7 since $1+x+y=(1+x)+(x+y)$. 

\underline{Case 15: $\angles{x, 1+y, x+y, 1+x+y}$}; this is same as Case 14. 

\underline{Case 16: $\angles{x, 1+x, 1+y, x+y, 1+x+y}$}; this is same as Case 12 since $1+x+y=(1+x)+(x+y)$.
\end{proof}

\begin{pro}
The only prime ideal on $\mathbb{B}[x,y]$, generated by linear polynomials, containing $y$ but not $x$ is $\angles{y}$.
\end{pro}
\begin{proof}
  This is symmetrical to Proposition \ref{proposition: x case}.
\end{proof}

\begin{pro}
 If $I$ is an ideal of $\mathbb{B}[x,y]$ generated by linear polynomials and $x, y \not \in I$, then $I$ is not a prime ideal. 
\end{pro}
\begin{proof}
\underline{Case 1: $\angles{1+x}$}; this is not prime by the exact same reason as in \cite{mincheva2025prime}. To be precise, 
\begin{equation}\label{eq: 1+x}
(1+x+x^3)(1+x^2+x^3)=(1+x)(1+x^2)(1+x^3),
\end{equation}
but $(1+x+x^2)$ and $(1+x^2+x^3)$ are not in $\angles{1+x}$.

\underline{Case 2: $\angles{1+y}$}; this is not prime by the same argument as in Case 1. 

\underline{Case 3: $\angles{x+y}$}; this is not a prime ideal. One may notice the following identity:
\begin{equation}\label{eq: alpha}
\alpha=xy(x+y+xy)(1+x+y)=(x+y)(x+xy)(y+xy).
\end{equation}
So, we have $\alpha \in \angles{x+y}$. But, clearly $xy, x+y+xy, 1+x+y$ are not in $\angles{x+y}$.\footnote{One can also argue that if $\angles{x+y}$ is a prime ideal, it's dehomogenization $\angles{1+x}$ is a prime ideal by Proposition \ref{pro: dehomogenization prime}, but $\angles{1+x}$ is not prime as shown in Case 1.} 

\underline{Case 4: $\angles{1+x+y}$}; this is not prime. In fact, from \eqref{eq: alpha}, we have that $\alpha \in \angles{1+x+y}$, but it is clear that $x+y, x+xy, y+xy$ are not in $\angles{1+x+y}$ since any nonzero element in $\angles{1+x+y}$ should have at least three monomials. 

\underline{Case 5: $\angles{1+x, 1+y}$}; this is not prime. In fact, we may look at \eqref{eq: 1+x}. Then we have $(1+x+x^3)(1+x^2+x^3) \in \angles{1+x,1+y}$, but clearly $(1+x+x^2)$ and $(1+x^2+x^3)$ are not in $\angles{1+x,1+y}$. 

\underline{Case 6: $\angles{1+x, x+y}$}; a similar argument as in Case 5 shows that $\angles{1+x,x+y}$ is not prime. 

\underline{Case 7: $\angles{1+x, 1+x+y}$}; a similar argument as in Case 5 shows that $\angles{1+x,x+y}$ is not prime.

\underline{Case 8: $\angles{1+y, x+y}$}; this is same as Case 6. 

\underline{Case 9: $\angles{1+y, 1+x+y}$}; this is same as Case 7. 

\underline{Case 10: $\angles{x+y, 1+x+y}$}; this is not prime. In fact, we have
\begin{equation}
(1+y)(x^2+xy+y) = (x+y+xy)(x+y) + (y+xy)(1+x+y) \in \angles{x+y,1+x+y}.
\end{equation}
It is clear that $1+y, x^2+xy+y \not \in \angles{x+y,1+x+y}$. 

\underline{Case 11: $\angles{1+x, 1+y, x+y}$}; a similar argument as in Case 5 shows that $\angles{1+x,x+y}$ is not prime.

\underline{Case 12: $\angles{1+x, 1+y, 1+x+y}$}; this is same as Case 5 since $1+x+y=(1+x)+(1+y)$. 

\underline{Case 13: $\angles{1+x, x+y, 1+x+y}$}; this is same as Case 6 since $1+x+y = (1+x)+(x+y)$.

\underline{Case 14: $\angles{1+y, x+y, 1+x+y}$}; this is same as Case 8 since $1+x+y = (1+y)+(x+y)$.

\underline{Case 15: $\angles{1+x, 1+y, x+y, 1+x+y}$}; this is same as Case 11 since $1+x+y = (1+x)+(1+y)$. 
\end{proof}

We conclude this section with the following conjecture.

\begin{conjecture}
Let $I$ be an ideal of $\mathbb{B}[x_1,\dots,x_n]$, generated by linear polynomials. If $I$ does not contain all of $x_1,\dots,x_n$, then $I$ is a prime ideal only when it is generated by a subset of variables $\{x_1,\dots,x_n\}$.
\end{conjecture}

\section{Decomposition of $\Spec (\mathbb{B}[x,y])$}\label{section: decomposition}

We use standard notation. For a semiring $R$, we denote by $\Spec R$ the set of prime ideals of $R$, equipped with the Zariski topology. For an ideal $I\subseteq R$, we denote by $V(I)$ the subset of $\Spec R$ consisting of the prime ideals containing $I$. We write $V(x)=V(\angles{x})$. Similarly, we denote by $D(x)$ the subset of $\Spec R$ consisting of the prime ideals of $R$ that do \emph{not} contain $x$. For further details, see \cite{jun2017vcech}, \cite{JMT20}, and the references therein.

We begin with the following well-known lemma. We include the proof for completeness.

\begin{lem}\label{lemma: extension and contraction}
Let $R$ be a semiring and let $S\subseteq R$ be a multiplicative subset. Extension and contraction give a bijection between prime ideals of $R$ disjoint from $S$ and prime ideals of $S^{-1}R$. Explicitly, if $P\subseteq R$ is prime and $P\cap S=\emptyset$, then $S^{-1}P$ is prime in $S^{-1}R$; if $\mathfrak q\subseteq S^{-1}R$ is prime, then its contraction $\mathfrak q^c=\{r\in R:r/1\in \mathfrak q\}$ is prime and disjoint from $S$; and these operations are inverse to each other.  
\end{lem}
\begin{proof}
If $P\subseteq R$ is prime and disjoint from $S$, then $S^{-1}P$ is proper. If $(a/s)(b/t)=ab/(st)\in S^{-1}P$, then there exists $u\in S$ such that $uab\in P$. Since $u\notin P$ and $P$ is prime, we get $ab\in P$, hence $a\in P$ or $b\in P$, so $a/s\in S^{-1}P$ or $b/t\in S^{-1}P$. Thus $S^{-1}P$ is prime. Conversely, if $\mathfrak q\subseteq S^{-1}R$ is prime, then its contraction $\mathfrak q^c$ is prime because $ab\in\mathfrak q^c$ implies $(a/1)(b/1)\in\mathfrak q$, hence $a/1\in\mathfrak q$ or $b/1\in\mathfrak q$. Also $\mathfrak q^c\cap S=\emptyset$, since every element of $S$ becomes a unit in $S^{-1}R$ and no proper ideal contains a unit. Finally, if $J\subseteq S^{-1}R$ is an ideal and $a/s\in J$, then $a/1=(s/1)(a/s)\in J$, so $a\in J^c$ and hence $a/s\in S^{-1}J^c$; this shows $J=S^{-1}J^c$. The remaining equality $(S^{-1}P)^c=P$ follows from the defining membership criterion and the assumption $P\cap S=\emptyset$.
\end{proof}

By Lemma \ref{lemma: extension and contraction}, $D(xy)$ corresponds to the prime ideals of the Laurent polynomial semiring $\mathbb B[x^{\pm1},y^{\pm1}]$ over $\mathbb B$. Moreover, in \cite[Theorem 5.1]{mincheva2025prime}, Mincheva and Sakran describe certain special subsets of $V(x,y)$ in terms of prime subsets of $\mathbb N^2$. In our setting, to understand $\Spec(\mathbb B[x,y])$, we focus on the pieces $V(x)\cap D(y)$ and $V(y)\cap D(x)$. By symmetry, it therefore suffices to study $V(x)\cap D(y)$.

Let $P\in V(x)\cap D(y)$, and let $S=\{1,y,y^2,\ldots\}$. Since $P\cap S=\emptyset$, localization gives a prime ideal $P_y=S^{-1}P\subseteq \mathbb B[x,y^{\pm1}]$ containing $x$. Conversely, primes of $\mathbb B[x,y^{\pm1}]$ containing $x$ contract to primes in $V(x)\cap D(y)$. Thus the problem may be reformulated as follows: classify the prime ideals of $\mathbb B[x,y^{\pm1}]$ containing $x$. Such a prime automatically contains every monomial $x^ay^b$ with $a>0$ and $b\in\mathbb Z$, because $x\in P_y$ and $y$ is invertible.

We view $\mathbb B[y^{\pm1}]$ as a subsemiring of $\mathbb B[x,y^{\pm1}]$. Every element $F\in \mathbb B[x,y^{\pm1}]$ has a unique decomposition 
\begin{equation}\label{eq: two parts}
F=a_F+h_F,\ a_F\in \mathbb B[y^{\pm1}],\ h_F\in x\left(\mathbb B[x,y^{\pm1}]\right).
\end{equation}
where $a_F$ is the $x$-free part of $F$ and every monomial of $h_F$ has positive $x$-degree. In what follows, we will use the notation \eqref{eq: two parts} without mentioning it each time. 

Consider the evaluation map at $x=0$:
\[
\rho:\mathbb B[x,y^{\pm1}]\to \mathbb B[y^{\pm1}],\quad \rho(a_F+h_F)=a_F.
\]
Clearly, $\rho$ is a semiring homomorphism. In fact, $\rho$ induces an injection $\rho^*:\Spec (\mathbb B[y^{\pm1}])\to \Spec ( \mathbb B[x,y^{\pm1}] )$ from the following.\footnote{This is also well-known, but we only include for the reader's convenience.}

\begin{lem}
Let $f:A \to B$ be a surjective semiring homomorphism. Then, the induced map $f^*:\Spec B \to \Spec A$ is an injection. 
\end{lem}
\begin{proof}
First, if $Q\subseteq B$ is prime, then $f^{-1}(Q)$ is prime in $A$. Indeed, it is an ideal because $f$ preserves addition and multiplication. It is proper since $1_A\in f^{-1}(Q)$ would imply $1_B=f(1_A)\in Q$, contradicting that $Q$ is proper. Finally, if $ab\in f^{-1}(Q)$, then $f(a)f(b)=f(ab)\in Q,$
so $f(a)\in Q$ or $f(b)\in Q$, hence $a\in f^{-1}(Q)$ or $b\in f^{-1}(Q)$.

Now suppose $Q_1,Q_2\in \Spec B$ and $f^{-1}(Q_1)=f^{-1}(Q_2)$. If $b\in Q_1$, choose $a\in A$ with $f(a)=b$, using surjectivity. Then $a\in f^{-1}(Q_1)=f^{-1}(Q_2)$, so $b=f(a)\in Q_2$. Thus $Q_1\subseteq Q_2$, and the reverse inclusion is identical. Hence $Q_1=Q_2$.
\end{proof}

We now consider the natural restriction map obtained by contraction. Let $\mathfrak P\in V(x)\cap D(y)$ and let $S=\{1,y,y^2,\ldots\}$. Since $\mathfrak P\cap S=\emptyset$, the localization $S^{-1}\mathfrak P$ is a prime ideal of $\mathbb B[x,y^{\pm1}]$. The inclusion $\mathbb B[y^{\pm1}]\hookrightarrow\mathbb B[x,y^{\pm1}]$ therefore gives a prime ideal
\[
(S^{-1}\mathfrak P)\cap\mathbb B[y^{\pm1}]
\]
by contraction. Hence there is a natural restriction map
\begin{equation}\label{eq: iota}
\iota:V(x)\cap D(y)\longrightarrow\operatorname{Spec}\mathbb B[y^{\pm1}],\qquad
\iota(\mathfrak P)=(S^{-1}\mathfrak P)\cap\mathbb B[y^{\pm1}].
\end{equation}

We will be interested in the fibers of $\iota$. To be precise, for a fixed prime ideal
$\mathfrak q\subseteq \mathbb B[y^{\pm1}]$ and $S=\{1,y,y^2,\ldots\}$, the fiber of the map
$\iota$ in \eqref{eq: iota} is
\begin{equation}
\mathcal F_{\mathfrak q}
:=
\left\{
\mathfrak P\in V(x)\cap D(y)
\;\middle|\;
(S^{-1}\mathfrak P)\cap \mathbb B[y^{\pm1}]
=
\mathfrak q
\right\}.
\end{equation}
$\mathcal{F}_q$ is a subset of $\Spec \mathbb{B}[x,y]$. It will be convenient to study this fiber after localizing at $S$, i.e., by viewing it as a subset of $\Spec \mathbb{B}[x,y^{\pm1}]$. To this end, we consider the following set.

\begin{mydef}
For a fixed prime ideal
$\mathfrak q\subseteq \mathbb B[y^{\pm1}]$, define the following set:
\begin{equation}
\widetilde{\mathcal F}_{\mathfrak q}
:=
\left\{
\mathfrak p\in
\operatorname{Spec}\mathbb B[x,y^{\pm1}]
\;\middle|\;
x\in\mathfrak p,\ 
\mathfrak p\cap\mathbb B[y^{\pm1}]
=
\mathfrak q
\right\}.
\end{equation}
\end{mydef}

By Lemma~\ref{lemma: extension and contraction}, we have an order-preserving bijection 
\[
\mathcal F_{\mathfrak q}
\to
\widetilde{\mathcal F}_{\mathfrak q}, \quad
\mathfrak P\longmapsto S^{-1}\mathfrak P, \quad
\mathfrak p\cap\mathbb B[x,y] \mapsfrom  \mathfrak p.
\]
Thus to determine the minimal and maximal elements of ${\mathcal F}_{\mathfrak q}$, it is enough to consider $\widetilde{\mathcal F}_{\mathfrak q}$. 

Recall that every
$F\in\mathbb B[x,y^{\pm1}]$ has a unique decomposition
\[
F=a_F+h_F, \quad
a_F\in\mathbb B[y^{\pm1}], \quad
h_F\in x\left(\mathbb B[x,y^{\pm1}]\right),
\]
and that
\[
\rho:\mathbb B[x,y^{\pm1}]
\longrightarrow
\mathbb B[y^{\pm1}], \quad
\rho(a_F+h_F)=a_F,
\]
is evaluation at $x=0$.

\begin{lem}\label{lemma: minimal}
For every prime ideal
$\mathfrak q\subseteq\mathbb B[y^{\pm1}]$, the ideal
\[
\widetilde{\mathfrak p}_{\min}(\mathfrak q)
:=
\rho^{-1}(\mathfrak q)
=
\left\{
F=a_F+h_F\in\mathbb B[x,y^{\pm1}]
\;\middle|\;
a_F\in\mathfrak q
\right\}
\]
belongs to $\widetilde{\mathcal F}_{\mathfrak q}$ and is its
minimal element with respect to inclusion.
\end{lem}

\begin{proof}
Since $\rho$ is a semiring homomorphism and $\mathfrak q$ is
prime, the inverse image
$\rho^{-1}(\mathfrak q)$ is a prime ideal. Moreover,
$\rho(x)=0\in\mathfrak q,$
so $x\in\widetilde{\mathfrak p}_{\min}(\mathfrak q)$, and
$\widetilde{\mathfrak p}_{\min}(\mathfrak q)
\cap\mathbb B[y^{\pm1}] =
\mathfrak q.$
Consequently,
$\widetilde{\mathfrak p}_{\min}(\mathfrak q)
\in\widetilde{\mathcal F}_{\mathfrak q}$.

Now let
$\mathfrak p\in\widetilde{\mathcal F}_{\mathfrak q}$, and suppose that
$F=a_F+h_F \in\widetilde{\mathfrak p}_{\min}(\mathfrak q).$
Then
$a_F\in\mathfrak q = \mathfrak p\cap\mathbb B[y^{\pm1}],$
so $a_F\in\mathfrak p$. Since $x\in\mathfrak p$, we also have
$h_F\in x\left(\mathbb B[x,y^{\pm1}]\right)
\subseteq\mathfrak p.$
It follows that $F=a_F+h_F\in\mathfrak p$. Hence
$\widetilde{\mathfrak p}_{\min}(\mathfrak q)
\subseteq\mathfrak p,$
as required.
\end{proof}

\begin{lem}\label{lemma: maximal}
For every prime ideal
$\mathfrak q\subseteq\mathbb B[y^{\pm1}]$, the set
\[
\widetilde{\mathfrak p}_{\max}(\mathfrak q)
:=
\left\{
F=a_F+h_F\in\mathbb B[x,y^{\pm1}]
\;\middle|\;
a_F\in\mathfrak q
\text{ or }
h_F\neq0
\right\}
\]
is a prime ideal belonging to
$\widetilde{\mathcal F}_{\mathfrak q}$. Moreover, it is the
maximal element of
$\widetilde{\mathcal F}_{\mathfrak q}$ with respect to inclusion.
\end{lem}

\begin{proof}
We first show that
$\widetilde{\mathfrak p}_{\max}(\mathfrak q)$ is an ideal. 

Let
$F=a_F+h_F, \
G=a_G+h_G$
belong to
$\widetilde{\mathfrak p}_{\max}(\mathfrak q)$. If either $h_F$
or $h_G$ is nonzero, then the positive $x$-degree part of
$F+G$ is nonzero. If $h_F=h_G=0$, then
$a_F,a_G\in\mathfrak q$, so
$F+G=a_F+a_G\in\mathfrak q.$
Thus the set is closed under addition.

Now let $F\in\widetilde{\mathfrak p}_{\max}(\mathfrak q)$ and
$G\in\mathbb B[x,y^{\pm1}]$. If $G=0$, then $FG=0$ belongs to
the set. Suppose $G\neq0$. If $h_F\neq0$, then $FG$ has a
nonzero term of positive $x$-degree, since
$\mathbb B[x,y^{\pm1}]$ has no zero divisors. If $h_F=0$, then
$a_F\in\mathfrak q$, and the $x$-free part of $FG$ is
$a_Fa_G\in\mathfrak q$. Hence $FG$ again belongs to
$\widetilde{\mathfrak p}_{\max}(\mathfrak q)$.

Its complement is
$\mathbb B[x,y^{\pm1}]
\setminus
\widetilde{\mathfrak p}_{\max}(\mathfrak q) = \mathbb B[y^{\pm1}]\setminus\mathfrak q.$
Because $\mathfrak q$ is prime, this complement is
multiplicatively closed. Therefore
$\widetilde{\mathfrak p}_{\max}(\mathfrak q)$ is prime.

It contains $x$, since the $x$-positive part of $x$ is
nonzero, and
$\widetilde{\mathfrak p}_{\max}(\mathfrak q)
\cap\mathbb B[y^{\pm1}] = \mathfrak q.$
Thus
$\widetilde{\mathfrak p}_{\max}(\mathfrak q) \in
\widetilde{\mathcal F}_{\mathfrak q}.$

Finally, let
$\mathfrak p\in\widetilde{\mathcal F}_{\mathfrak q}$ and let
$F=a_F+h_F\in\mathfrak p$. If $h_F\neq0$, then
$F\in\widetilde{\mathfrak p}_{\max}(\mathfrak q)$ by definition.
If $h_F=0$, then
$F=a_F \in \mathfrak p\cap\mathbb B[y^{\pm1}] = \mathfrak q,$
so again
$F\in\widetilde{\mathfrak p}_{\max}(\mathfrak q)$. Hence
$\mathfrak p
\subseteq
\widetilde{\mathfrak p}_{\max}(\mathfrak q),$
which proves maximality.
\end{proof}

Combining Lemmas~\ref{lemma: minimal} and
\ref{lemma: maximal}, and then applying the
extension--contraction correspondence, gives the following theorem.

\begin{mythm}\label{thm: boundary-fiber}
Let
$\mathfrak q\in\operatorname{Spec}\mathbb B[y^{\pm1}]$. Then the
fiber
\[
\mathcal F_{\mathfrak q}
=
\left\{
\mathfrak P\in V(x)\cap D(y)
\;\middle|\;
(S^{-1}\mathfrak P)\cap\mathbb B[y^{\pm1}]
=
\mathfrak q
\right\}
\]
has a minimal element and a maximal element with respect to
inclusion. More precisely, these elements are
$\mathfrak P_{\min}(\mathfrak q)
=
\widetilde{\mathfrak p}_{\min}(\mathfrak q)
\cap\mathbb B[x,y]$
and
$\mathfrak P_{\max}(\mathfrak q)
=
\widetilde{\mathfrak p}_{\max}(\mathfrak q)
\cap\mathbb B[x,y].$
If
\[
F=a_F+h_F\in\mathbb B[x,y], \quad
a_F\in\mathbb B[y], \quad
h_F\in x\left(\mathbb B[x,y]\right),
\]
then they admit the explicit descriptions
\[
\mathfrak P_{\min}(\mathfrak q) = \left\{F=a_F+h_F
\mid
a_F\in\mathfrak q\cap\mathbb B[y]\right\},
\]
and
\[
\mathfrak P_{\max}(\mathfrak q) = \left\{F=a_F+h_F \mid a_F\in\mathfrak q\cap\mathbb B[y] \textrm{ or } h_F\neq0 \right\}.
\]
Consequently, every
$\mathfrak P\in\mathcal F_{\mathfrak q}$ satisfies
$\mathfrak P_{\min}(\mathfrak q)
\subseteq
\mathfrak P
\subseteq
\mathfrak P_{\max}(\mathfrak q).$
\end{mythm}

\begin{rmk}\label{rmk: subtractive boundary fiber}
The map $\iota:V(x)\cap D(y)\to\operatorname{Spec}\mathbb B[y^{\pm1}]$ restricts to a bijection between subtractive prime ideals (Definition \ref{definition: ideals and primes}) in $V(x)\cap D(y)$ and subtractive prime ideals of $\mathbb B[y^{\pm1}]$. Its inverse sends $\mathfrak q$ to
$$
\mathfrak P_{\min}(\mathfrak q)=\rho^{-1}(\mathfrak q)\cap\mathbb B[x,y]=\{F=a_F+h_F\in\mathbb B[x,y] \mid a_F\in\mathfrak q\cap\mathbb B[y]\},
$$
where $\rho:\mathbb B[x,y^{\pm1}]\to\mathbb B[y^{\pm1}]$ is evaluation at $x=0$ and $h_F\in x \left(\mathbb B[x,y]\right)$.

Indeed, let $\mathfrak P\in V(x)\cap D(y)$ be subtractive, set $\mathfrak q=\iota(\mathfrak P)$, and write $\mathfrak p=S^{-1}\mathfrak P$. Lemma~\ref{lemma: minimal} gives $\rho^{-1}(\mathfrak q)\subseteq\mathfrak p$. Conversely, if $F=a_F+h_F\in\mathfrak p$, then after multiplying by a sufficiently large power $y^N$ we have $y^NF\in\mathfrak P$, with $y^Nh_F\in\mathfrak P$ because $x\in\mathfrak P$. Subtractivity therefore gives $y^Na_F\in\mathfrak P$, so $a_F\in\mathfrak q$ and hence $F\in\rho^{-1}(\mathfrak q)$. Thus $\mathfrak p=\rho^{-1}(\mathfrak q)$ and $\mathfrak P=\mathfrak P_{\min}(\mathfrak q)$. Conversely, if $\mathfrak q$ is subtractive and prime, then $\rho^{-1}(\mathfrak q)$ and its contraction to $\mathbb B[x,y]$ are subtractive and prime. Hence every fiber over a subtractive prime contains exactly one subtractive prime, namely its minimal element. The possible intermediate primes arise from the failure of subtractivity.
\end{rmk}

\begin{rmk}
\label{rmk: multivariable boundary fiber}
Theorem \ref{thm: boundary-fiber} extends without change to $R_m[x]:=\mathbb B[y_1^{\pm1},\ldots,y_m^{\pm1}]$. For $\mathfrak q\in\operatorname{Spec}R_m$, the fiber
$$
\widetilde{\mathcal F}_{\mathfrak q}=\{\mathfrak p\in\operatorname{Spec}R_m[x] \mid x\in\mathfrak p,\ \mathfrak p\cap R_m=\mathfrak q\}
$$
has minimal and maximal elements:
\[
\mathfrak p_{\min}(\mathfrak q)=\{a+h \mid a\in\mathfrak q,\ h\in x\left(R_m[x]\right)\} \quad \text{and} \quad \mathfrak p_{\max}(\mathfrak q)=\{a+h \mid a\in\mathfrak q\text{ or }h\neq0\}.
\]
Equivalently, if $S$ is the multiplicative subset of $\mathbb B[x,y_1,\ldots,y_m]$ generated by $y_1,\ldots,y_m$, then the fibers of the localization--contraction map
$$
V(x)\cap D(y_1\cdots y_m)\longrightarrow\operatorname{Spec}\mathbb B[y_1^{\pm1},\ldots,y_m^{\pm1}],\qquad \mathfrak P\longmapsto(S^{-1}\mathfrak P)\cap\mathbb B[y_1^{\pm1},\ldots,y_m^{\pm1}]
$$
have minimal and maximal elements obtained by contracting the ideals above. The proof is identical, since it uses only the unique decomposition $F=a_F+h_F$ with $a_F\in R_m$ and $h_F\in x\left(R_m[x]\right)$, evaluation at $x=0$, and primeness of $\mathfrak q$. The same argument applies after choosing any coordinate as a fixed variable and inverting any collection of the remaining coordinates.
\end{rmk}

\begin{rmk}
Theorem~\ref{thm: boundary-fiber} gives lower and upper
bounds for the primes in the fiber, but it does not explicitly
describe the possible intermediate prime ideals. Such intermediate
primes encode mixed conditions involving both $x$ and $y$.

It is most natural to record these conditions in the localized
fiber, where $y$ is invertible. For
$\mathfrak p\in\widetilde{\mathcal F}_{\mathfrak q}$, define
\[
B_x(\mathfrak p)
:=
\left\{
(a,b)\in\mathbb N_{>0}\times\mathbb Z
\;\middle|\;
1+x^ay^b\in\mathfrak p
\right\}.
\]
Then $B_x(\mathfrak p)$ is a prime subset of the monoid
$\mathbb N_{>0}\times\mathbb Z$ (Definition \ref{definition: prime subset}).

Indeed, let
$u=(a,b), \ v=(c,d), \ X^u=x^ay^b, \ X^v=x^cy^d.$
Suppose that $u+v\in B_x(\mathfrak p)$. Then
$1+X^{u+v}\in\mathfrak p.$
Since $x\in\mathfrak p$ and the first coordinates of $u$ and
$v$ are positive, both $X^u$ and $X^v$ belong to
$\mathfrak p$. Therefore,
\[
(1+X^u)(1+X^v)=1+X^u+X^v+X^{u+v} \in\mathfrak p.
\]
By primeness of $\mathfrak p$,
$(1+X^u)\in\mathfrak p$ or 
$(1+X^v)\in\mathfrak p.$
Thus
$u\in B_x(\mathfrak p)
\ \text{or}\
v\in B_x(\mathfrak p),$
which proves that $B_x(\mathfrak p)$ is prime.

For
$\mathfrak P\in\mathcal F_{\mathfrak q}$, one may equivalently
define
$B_x(\mathfrak P)
:=
B_x(S^{-1}\mathfrak P).$
This invariant is analogous to the one-variable invariant
$A_P=\{a\mid 1+t^a\in P\}$
introduced in \cite{alarcon1994commutative} and further studied in
\cite{mincheva2025prime}. It need not determine the prime ideal by
itself, since relations involving polynomials with larger support may
contain additional information.
\end{rmk}

\section{Maps from $\mathbb{B}[x,y]$ to $\mathbb{B}[t]$}

In this section, we consider the following maps. For $i,j \in \mathbb{N}$, let
\begin{equation}\label{eq: pi map}
\pi_{i,j}:\mathbb{B}[x,y] \to \mathbb{B}[t], \qquad x\mapsto t^i,\quad y\mapsto t^j.
\end{equation}

For $F \neq 0\in \mathbb{B}[x,y]$, we let $T_{a,b}(F)=\{ai+bj\mid (i,j)\in\operatorname{Supp}(F)\}$ and consider the following:
\[
\mu_{a,b}(F)=\min T_{a,b}(F), \quad T'_{a,b}(F):=\{n - \mu_{a,b}(F) \mid n \in T_{a,b}(F)\}.
\]

Two specific classes of prime ideals, $I_A$ and $J_A$, were introduced in \cite{alarcon1994commutative} and studied further in \cite{mincheva2025prime}. For $A\subseteq\mathbb{N}$, they are defined by
\begin{equation}\label{eq: JA}
I_A:=\angles{t,{1+t^a\mid a\in A}},\qquad
J_A:=I_A-\{t^m f\mid m\geq 0,\ f\notin I_A\}.
\end{equation}

Our goal is to give an explicit description of the preimages of $I_A$ and $J_A$ under $\pi_{i,j}$. This description will be used later to show that, in our setting, the homogenization of a prime ideal need not be prime (see Example \ref{example: homogenization}).

\begin{lem}\label{lemma: one-variable lemma}
Let $A\subseteq \mathbb{N}$ be a prime subset. Let $f \neq 0 \in \mathbb{B}[t]$ and $E=\text{Supp}(f)$. Then, we have $f \in I_A$ if and only if $0 \not \in E$ or $E \cap A \neq \emptyset$.
\end{lem}
\begin{proof}
Suppose that $f\in I_A$ and $0 \in E$. Since $I_A$ is generated by $t$ and the polynomials $1+t^c$ with $c\in A$, we can write
\[
f=h_0(t)t+\sum_{\ell=1}^r h_\ell(t)(1+t^{c_\ell})
\]
for some $h_0,h_1,\ldots,h_r\in\mathbb B[t]$ and $c_1,\ldots,c_r\in A$. The summand $h_0(t)t$ contributes only positive powers of $t$, so it cannot contribute the constant term of $f$. Since $0\in E$, the constant term comes from $h_\ell(t)(1+t^{c_\ell})$ for some $\ell$. It follows that $0 \in \text{Supp}(h_\ell(t))$ for some $\ell$. But then the term $h_\ell(t)(1+t^{c_\ell})$ also contributes the monomial $t^{c_\ell}$. Since addition in $\mathbb B[t]$ is idempotent, supports are combined by union and there is no cancellation. Therefore $t^{c_\ell}$ remains in the support of $f$. Hence $c_\ell\in E\cap A$, so $E\cap A\neq\emptyset$. 

Conversely, suppose first that $0\notin E$. Then every monomial of $f$ is divisible by $t$, so $f\in\langle t\rangle\subseteq I_A$. Now suppose that $E\cap A\neq\emptyset$. If $0\notin E$, we are done by the previous case. If $0\in E$, choose $c\in E\cap A$. Then, we let
\[
f=(1+t^c)+\sum_{n\in E\setminus\left\{0,c\right\}}t^n.
\]
Each monomial $t^n$ appearing in the remaining sum has $n>0$, and hence lies in $\langle t\rangle\subseteq I_A$. Also $1+t^c\in I_A$ by definition. Thus $f\in I_A$.
\end{proof}

\begin{lem}\label{lemma: membership lemma}
With the same notation as above, 
We have the following:
\begin{enumerate}
\item %
$\pi_{a,b}^{-1}(I_A)=\left\{F\in\mathbb B[x,y]\setminus\left\{0\right\} \mid 0\notin T_{a,b}(F)\text{ or }T_{a,b}(F)\cap A\neq\emptyset\right\}\cup\left\{0\right\}$,
and 
\item %
$\pi_{a,b}^{-1}(J_A)=\left\{F\in\mathbb B[x,y]\setminus\left\{0\right\}\mid T'_{a,b}(F)\cap A\neq\emptyset\right\}\cup\left\{0\right\}$.
\end{enumerate}
\end{lem}

\begin{proof}
(1) Let
\[
F=\sum_{(i,j)\in\operatorname{Supp}(F)}x^iy^j\neq 0.
\]
Then, we have
\[
\pi_{a,b}(F)=\sum_{(i,j)\in\operatorname{Supp}(F)}t^{ai+bj}.
\]
Note that different monomials of $F$ may map to the same power of $t$, but repeated exponents collapse because addition in $\mathbb B[t]$ is idempotent. Hence, we have
\begin{equation}\label{eq: support}
\operatorname{Supp}(\pi_{a,b}(F))=T_{a,b}(F).
\end{equation}
By Lemma \ref{lemma: one-variable lemma}, applied to $f=\pi_{a,b}(F)$, we have
\begin{equation}\label{eq: I_A}
F\in\pi_{a,b}^{-1}(I_A) 
\iff 0\notin\operatorname{Supp}(\pi_{a,b}(F))\text{ or }\operatorname{Supp}(\pi_{a,b}(F))\cap A\neq\emptyset.
\end{equation}
From \eqref{eq: support}, this is equivalent to
\begin{equation}
0\notin T_{a,b}(F)\text{ or }T_{a,b}(F)\cap A\neq\emptyset.  
\end{equation}
This proves the first assertion.

(2) Let $F\neq 0 \in \mathbb{B}[x,y]$. 
We let $\mu=\mu_{a,b}(F)$. Since $\operatorname{Supp}(\pi_{a,b}(F))=T_{a,b}(F)$, we can factor
\[
\pi_{a,b}(F)=t^\mu G_F(t),\quad \text{where} \quad G_F(t)=\sum_{n\in T_{a,b}(F)}t^{n-\mu}.
\]
By definition, $0 \in \text{Supp}(G_F)$, and in fact, we have
\[
\operatorname{Supp}(G_F)=T_{a,b}(F)-\mu=T'_{a,b}(F).
\]
By the definition of $J_A$, the polynomial $\pi_{a,b}(F)$ lies in $J_A$ if and only if $G_F\in I_A$. Since $0\in\operatorname{Supp}(G_F)$, by Lemma \ref{lemma: one-variable lemma}, we have
\[
G_F\in I_A \iff \operatorname{Supp}(G_F)\cap A\neq\emptyset.
\]
Since $\operatorname{Supp}(G_F)=T'_{a,b}(F)$, we have
\[
F\in\pi_{a,b}^{-1}(J_A) \iff T'_{a,b}(F)\cap A\neq\emptyset.
\]
\end{proof}
We note that in a view of Section \ref{section: decomposition}, we may observe that for $a,b >0$, 
\[
\pi_{a,b}^{-1}(I_A) \in V(x,y) \quad \textrm{and }\quad \pi_{a,b}^{-1}(J_A)\in D(xy).
\]

\section{Homogenization and dehomogenization}\label{section: homogenization}

In this section, we consider the homogenization and dehomogenization of prime ideals. We begin by recalling the corresponding ring-theoretic statements. Let $B=\bigoplus_{d\geq 0}B_d$ be a $\mathbb{Z}_{\geq 0}$-graded algebra over a ring $A$, and let $I\subseteq B$ be an ideal. The associated homogeneous ideal is
\[
I^h=\bigoplus_{d\geq 0}(I\cap B_d),
\]
or, equivalently, the ideal consisting of finite sums of homogeneous elements of $I$. Thus, $I$ is homogeneous if and only if $I=I^h$.

The following are elementary facts for rings.

\begin{lem}
Let $I,J\subseteq B$ be ideals of a graded ring $B$. Then the following hold.
\begin{enumerate}
\item[(a)] If $I$ is prime, then the associated homogeneous ideal $I^h$ is prime.
\item[(b)] Let us suppose $I$ and $J$ are homogeneous. Then $V_+(I)\subseteq V_+(J)$ if and only if $J\cap B_+\subseteq \sqrt{I}$.
\item[(c)] We have $\operatorname{Proj} B=\varnothing$ if and only if $B_+$ is nilpotent.
\end{enumerate}
\end{lem}

We now ask whether the same statement holds for prime ideals in $\mathbb B[x,y]$. Let's equip $S=\mathbb B[x,y]$ with the standard grading by total degree, and let $S_d$ be the set of homogeneous polynomials of degree $d$. If $I\subseteq S$ is an ideal, define its associated homogeneous ideal by 
\[
I^h:=\bigoplus_{d\geq 0}(I\cap S_d).
\]
Hence, $F=\sum_d F_d \in I^h$ if and only if every homogeneous component $F_d$ lies in $I$.

Different from the classical case for rings, the associated homogeneous ideal of a prime ideal need not be prime for semirings as the following example shows. 

\begin{example}\label{example: homogenization}
Let $A=\{1\}\subseteq \mathbb N$. Consider the following map:
\begin{equation}
\pi_{1,2}:\mathbb B[x,y]\to \mathbb B[t], \quad x\mapsto t, \quad y \mapsto t^2.  
\end{equation}
Since $A$ is a prime subset, the ideal $J_A\subseteq \mathbb B[t]$ as in \eqref{eq: JA} is prime. To be precise, from \eqref{eq: J_A}, one has
\[
J_A=\angles{\{1+t+t^m \mid m\geq 0\}}.
\]
Hence its contraction $\mathfrak P=\pi_{1,2}^{-1}(J_A)$ is a prime ideal of $\mathbb B[x,y]$. For a nonzero polynomial $F \in \mathbb{B}[x,y]$, set
\begin{equation}
 T_{1,2}(F)=\{i+2j \mid (i,j)\in \operatorname{Supp}(F)\}\quad \text{and} \quad \mu_{1,2}(F)=\min T_{1,2}(F).  
\end{equation}
It follows from Lemma \ref{lemma: membership lemma}, we have 
\begin{equation}
F\in\mathfrak P \iff 1 \in \{a-\mu_{1,2}(F) \mid a \in T_{1,2}(F)\}.
\end{equation}
Equivalently, $F\in\mathfrak P$ if and only if $T_{1,2}(F)$ contains two values differing by $1$, one of which is the minimum.

Now let $H\in \mathbb B[x,y]_d$ be homogeneous of total degree $d$, say
\[
H=\sum_{i\in S}x^iy^{d-i}
\]
for some $S\subseteq \{0,1,\ldots,d\}$. Then, we have
\[
T_{1,2}(H)=\{i+2(d-i)\mid i\in S\}=\{2d-i\mid i\in S\},
\]
and hence we have
\[
\mu_{1,2}(H)=2d-\max S\quad \text{and} \quad \{a-\mu_{1,2}(H) \mid a \in T_{1,2}(H)\}=\{\max S-i \mid i\in S\}.
\]
Therefore $H\in\mathfrak P$ if and only if $\max S-1\in S$. Thus the homogeneous ideal
\[
\mathfrak P^h=\bigoplus_{d\geq 0}(\mathfrak P\cap \mathbb B[x,y]_d)
\]
is characterized by the condition that every nonzero homogeneous component contains its top two $x$-exponents consecutively. More explicitly, with $\mathfrak P\cap \mathbb B[x,y]_0=\{0\}$,
\[
\mathfrak P\cap \mathbb B[x,y]_d=\big\{\sum_{i\in S}x^iy^{d-i} \mid S\subseteq\{0,\ldots,d\},\ \max S-1\in S\cup\{0\}\big\}.
\]
For example, $x+y\in\mathfrak P^h$ and $x^2+xy\in\mathfrak P^h$, whereas $x^2+y^2\notin\mathfrak P^h$.

We now show that $\mathfrak P^h$ is not prime. Let $f=1+x+y$ and $g=x+y+y^2$. Then $f\notin\mathfrak P^h$, since its degree-zero component is $1$ and $1\notin\mathfrak P$. Also, $g\notin\mathfrak P^h$, since its degree-two component is $y^2$ and $y^2\notin\mathfrak P$. However, in $\mathbb B[x,y]$ we have
\[
fg=x+y+x^2+xy+y^2+xy^2+y^3.
\]
Its homogeneous components are $x+y$, $x^2+xy+y^2$, and $xy^2+y^3$, and each satisfies the consecutive-top-exponent condition above. Hence $fg\in\mathfrak P^h$, while $f,g\notin\mathfrak P^h$. Therefore $\mathfrak P^h$ is not prime.
\end{example}

The failure occurs exactly at the point where the ring proof uses subtraction. In the ring proof, if $a=\sum_i a_i$ is not in $I^h$, one can pass modulo $I^h$ and discard any homogeneous components already lying in $I^h$; in particular, one may arrange that the top nonzero homogeneous component of the chosen representative is not in $I^h$. This is done by subtracting off the unwanted homogeneous components. In $\mathbb B[x,y]$ one cannot subtract such components. Thus $a\notin I^h$ only says that at least one homogeneous component of $a$ is not in $I$, not that the top homogeneous component has this property. The leading-term argument for primeness of $I^h$ therefore breaks down.

If we only consider ideals which are subtractive, then the original proof used in the case of rings still holds, so we see the homogenization of a subtractive prime ideal is prime.

\begin{pro}
Let $S$ be a graded semiring and $I$ be a subtractive prime ideal of $S$. Then the associated homogeneous ideal $I^h$ is a prime ideal. 
\end{pro}
\begin{proof}
The set $I^h$ is a homogeneous ideal and $I^h\subseteq I$, so $I^h$ is proper. Let $a,b\in S$ with $ab\in I^h$, and suppose for contradiction that $a,b\notin I^h$. Write $a=\sum_{i=0}^n a_i$ and $b=\sum_{j=0}^m b_j$, with $a_i\in S_i$ and $b_j\in S_j$. Since $a,b\notin I^h$, we may choose 
\[
r=\max\{i \mid a_i\notin I\} \quad \text{and} \quad s=\max\{j \mid b_j\notin I\}.
\]
Then $a_i\in I$ for all $i>r$ and $b_j\in I$ for all $j>s$. The homogeneous component of $ab$ of degree $(r+s)$ is
\begin{equation}
(ab)_{r+s}=a_rb_s+c,\quad \text{where } c:=\sum_{{i+j=r+s;(i,j)\neq(r,s)}}a_i b_j.
\end{equation}
For every summand $a_i b_j$ appearing in $c$, either $i>r$ or $j>s$, hence $a_i\in I$ or $b_j\in I$, and therefore $a_i b_j\in I$; thus $c\in I$. Since $ab\in I^h$, its degree $(r+s)$ homogeneous component lies in $I$, so $a_rb_s+c\in I$. Since $I$ is subtractive and $c\in I$, it follows that $a_rb_s\in I$. Since $I$ is prime, either $a_r\in I$ or $b_s\in I$, contradicting the choice of $r$ and $s$. Hence $a\in I^h$ or $b\in I^h$, and so $I^h$ is prime.
\end{proof}

We remark here that by a similar argument the associated homogeneous ideal $I^h$ of a prime ideal, given by homogenizing elements in the classical way, need not be prime either.

\begin{example}
\label{example: homogen2}
Let $A={1}\subseteq\mathbb N$, let $J_A\subseteq\mathbb B[t]$ be the corresponding prime ideal, and consider $ \pi_{1,2}:\mathbb B[x,y]\longrightarrow\mathbb B[t],\ x\mapsto t,\ y\mapsto t^2.$ Then $ \mathfrak P:=\pi_{1,2}^{-1}(J_A) $ is a prime ideal of $\mathbb B[x,y]$. Introduce a new variable $x_0$ and let $ \mathfrak P^{\operatorname{hom}} := \left\langle F^{\operatorname{hom}}\mid F\in\mathfrak P\right\rangle \subseteq\mathbb B[x_0,x,y]. $

Set $ f=1+x+y,\ g=x+y+y^2.$ As in Example~\ref{example: homogenization}, $ fg=(x+y)+(x^2+xy+y^2)+(xy^2+y^3), $ and each of the three displayed homogeneous components lies in $\mathfrak P$. Since they are already homogeneous, they lie in $\mathfrak P^{\operatorname{hom}}$, and hence $ fg\in\mathfrak P^{\operatorname{hom}}. $

On the other hand, $\mathfrak P^{\operatorname{hom}}$ is homogeneous. Thus, if $f\in\mathfrak P^{\operatorname{hom}}$, then its degree-zero component $1$ would lie in $\mathfrak P^{\operatorname{hom}}$. Likewise, if $g\in\mathfrak P^{\operatorname{hom}}$, then its degree-two component $y^2$ would lie in $\mathfrak P^{\operatorname{hom}}$. Dehomogenizing by setting $x_0=1$ sends $\mathfrak P^{\operatorname{hom}}$ into $\mathfrak P$, which would imply respectively that $1\in\mathfrak P$ or $y^2\in\mathfrak P$, both impossible. Therefore $f,g\notin\mathfrak P^{\operatorname{hom}}, \  fg\in\mathfrak P^{\operatorname{hom}}, $ so $\mathfrak P^{\operatorname{hom}}$ is not prime. 
\end{example}

This, being the construction used in the classical way of constructing the projective closure $V(P) \rightsquigarrow V_+(P^{hom})$, provides motivation for the use of congruences rather than ideals, as mentioned in the introduction.

We next record a positive result in the opposite direction: dehomogenizing a homogeneous prime preserves primeness.

\begin{pro}
\label{pro: dehomogenization prime}
Let
$S=\mathbb B[x_0,\ldots,x_n]$
with the standard grading, and let $P\subseteq S$ be a homogeneous
prime ideal. Fix $i\in\{0,\ldots,n\}$ and assume that $x_i\notin P$.
Set
$A_i:=\mathbb B[x_0,\ldots,\widehat{x_i},\ldots,x_n],$
and let
$\Phi_i:S\longrightarrow A_i$
be the dehomogenization homomorphism determined by
$\Phi_i(x_i)=1
\ \text{and}\ 
\Phi_i(x_j)=x_j , \ (j\neq i).$
Then $\Phi_i(P)$ is a prime ideal of $A_i$.
\end{pro}

\begin{proof}
Since $\Phi_i$ is surjective, the image $\Phi_i(P)$ is an ideal
of $A_i$. Indeed, if
$f=\Phi_i(F),\ g=\Phi_i(G)$
for $F,G\in P$, then
$f+g=\Phi_i(F+G)\in\Phi_i(P).$
Moreover, if $h\in A_i$, choose $H\in S$ with
$\Phi_i(H)=h$. Then
$hf=\Phi_i(HF)\in\Phi_i(P).$

We first show that $\Phi_i(P)$ is proper. Suppose, toward a
contradiction, that $1\in\Phi_i(P)$. Then there exists $F\in P$
such that
$\Phi_i(F)=1.$
Write $F$ as a sum of its homogeneous components:
$F=\sum_{d\geq 0}F_d.$
Since $P$ is homogeneous, every homogeneous component $F_d$
belongs to $P$.

The equality $\Phi_i(F)=1$ implies that every monomial appearing in
$F$ is a pure power of $x_i$. Consequently, every nonzero
homogeneous component $F_d$ is equal to $x_i^d$. Since $F\neq0$,
we therefore have
$x_i^d\in P$
for some $d\geq0$. If $d=0$, then $1\in P$, contradicting the
properness of $P$. If $d>0$, primeness of $P$ implies
$x_i\in P$, contrary to the hypothesis. Hence $\Phi_i(P)$ is
proper.

Now suppose that
$fg\in\Phi_i(P)$
for $f,g\in A_i$. If $f=0$ or $g=0$, then one of the factors
already belongs to $\Phi_i(P)$, so assume that $f\neq0$ and
$g\neq0$. Choose $F\in P$ such that
$\Phi_i(F)=fg.$
Write
\[
F=\sum_{d\in D}F_d,
\]
where each $F_d$ is homogeneous of degree $d$ and nonzero. Since
$P$ is homogeneous, $F_d\in P$ for every $d\in D$. Let
\[
N=\max D \quad \textrm{and} \quad H=\sum_{d\in D}x_i^{N-d}F_d,
\]
in particular, $H$ is homogeneous
of degree $N$. Moreover, $H\in P$, and
\[
\Phi_i(H) = \sum_{d\in D}\Phi_i(F_d) = \Phi_i(F) = fg.
\]

Let
$r=\deg(f),\ s=\deg(g),$
and let $f^{h_i}$ and $g^{h_i}$ denote their homogenizations with
respect to $x_i$. Explicitly, with vector notation, if
$f=\sum_{\alpha}\textbf{x}^\alpha,$
where the variables in $\textbf{x}^\alpha$ exclude $x_i$, then
\[
f^{h_i} = \sum_{\alpha}
x_i^{r-|\alpha|}\textbf{x}^\alpha,
\]
where $|\alpha|=\alpha_0+\dots + \alpha_n$ if $\alpha=(\alpha_0,\dots,\alpha_n)$, and similarly for $g^{h_i}$.

Since $f$ and $g$ are nonzero,
$\deg(fg)=r+s,$
and homogenization is multiplicative:
$(fg)^{h_i}=f^{h_i}g^{h_i}.$
Also, a homogeneous polynomial of degree $N$ is uniquely determined
by its dehomogenization with respect to $x_i$. Since $H$ is
homogeneous of degree $N$ and dehomogenizes to $fg$, it follows
that
\[
H = x_i^{N-r-s}(fg)^{h_i} = x_i^{N-r-s}f^{h_i}g^{h_i}.
\]
In particular,
$x_i^{N-r-s}f^{h_i}g^{h_i}\in P.$

Because $x_i\notin P$, no positive power of $x_i$ belongs to
$P$; moreover, $1\notin P$. Thus
$x_i^{N-r-s}\notin P.$
Primeness of $P$ therefore gives
$f^{h_i}g^{h_i}\in P.$
Applying primeness once more, we obtain
$f^{h_i}\in P
\ \text{or}\ 
g^{h_i}\in P.$
Dehomogenizing yields
$f=\Phi_i(f^{h_i})\in\Phi_i(P)
\ \text{or}\ 
g=\Phi_i(g^{h_i})\in\Phi_i(P).$
Therefore $\Phi_i(P)$ is prime.
\end{proof}

\section{Projective closure for Ideals and Congruences}

In this section, we construct a projective space $\mathbb{P}_{\mathbb{B}}^n$ over $\mathbb{B}$ by using prime congruences. We then consider the projective closure in this setting. We will only view them as topological spaces, i.e., we do not consider structure sheaves on them. In what follows, all polynomial semirings $S=\mathbb B[x_0,x_1,\ldots,x_n]$ will be considered with the standard grading $\deg(x_i)=1$. We begin with the following definition. We emphasize that much of this section generalized to the case of an idempotent semifield $K$ instead of $\mathbb{B}$, but we restrict ourselves to the case of $\mathbb{B}$ to avoid unnecessary technical complication.

\begin{mydef}\label{definition: U_i}
For $S=\mathbb B[x_0,\ldots,x_n]$ with the standard grading, define the standard affine congruence charts by
$U_i:=\operatorname{SpecCong}((S_{x_i})_0),\ 0\leq i\leq n,$
where $(S_{x_i})_0$ is the degree-zero localization as in Definition \ref{definition: degree-zero}.
\end{mydef}

For $S=\mathbb B[x_0,\ldots,x_n]$ the degree-zero localization $(S_{x_i})_0$ is naturally isomorphic to the semiring
$\mathbb B\left[\frac{x_0}{x_i},\ldots,\widehat{\frac{x_i}{x_i}},\ldots,\frac{x_n}{x_i}\right],$
so each $U_i$ is just $\mathbb{A}_{\mathbb{B}}^n$, defined in Definition \ref{definition: prime congreunce spectra}. In particular, on the distinguished chart $U_0$ one has
\begin{equation}\label{eq: identification}
(S_{x_0})_0\cong\mathbb B\left[\frac{x_1}{x_0},\ldots,\frac{x_n}{x_0}\right]\cong\mathbb B[t_1,\ldots,t_n], \quad \text{where } t_j=\frac{x_j}{x_0}.
\end{equation}

Now, we construct the projective space $\Proj$ (as a topological space) by properly gluing these charts. For $i\neq j$, the common overlap of the $i$th and $j$th charts is
\[
U_{ij}:=\operatorname{SpecCong}((S_{x_ix_j})_0).
\]
Equivalently, $U_{ij}$ is obtained from the $i$th chart by inverting the degree-zero coordinate $x_j/x_i$, or from the $j$th chart by inverting $x_i/x_j$. The two descriptions agree under the identification
\begin{equation}
 \frac{x_i}{x_j}=\left(\frac{x_j}{x_i}\right)^{-1}.   
\end{equation}
More generally, the corresponding identifications on higher intersections are compatible because iterated localization is independent of the order of localization. We define the congruence-theoretic analogue of projective $n$-space by gluing the affine congruence spectra $U_0,\ldots,U_n$ along these standard overlaps:

\begin{mydef}
With the same notation as above, we define
\[
\Proj:=(\bigsqcup_{i=0}^n U_i)/\sim,
\]
with $U_i$ and $U_j$ identified along $U_{ij}$ by the localization isomorphisms above, with the quotient topology. 
\end{mydef}

By Theorem~\ref{theorem: Tanaka} (2), each subset $D(x_j/x_i)$ of $U_i$ is open and the overlap
$U_{ij}$ is identified with
$D(x_j/x_i)\subseteq U_i$, and similarly with
$D(x_i/x_j)\subseteq U_j$. Consequently, the images of the standard
charts $U_i$ form an open cover of $\Proj$.

Now, we define the projective closure of an ideal $P \subseteq \mathbb{B}[x_1,\dots,x_n]$ in $\Proj$. We first need a couple of lemmas. 

For an ideal $P\subseteq \mathbb{B}[x_1,\dots,x_n]$, we denote the homogeneous bend congruence of $P$ by
$\operatorname{Bend}(P)^h$, recalled in Section \ref{subsection: congrouence}. As in the ring case, we saturate $\operatorname{Bend}(P)^h$  with respect to $x_0$. 

\begin{mydef}
If $C$ is a congruence on $S=\mathbb{B}[x_0,\dots,x_n]$ and $x_i$ is one of the homogeneous coordinates, define
\begin{equation}
 (C:x_i^\infty):=\{(F,G)\in S\times S:\exists N\geq0\text{ such that }(x_i^NF,x_i^NG)\in C\}.  
\end{equation}
\end{mydef}

The following shows that $(C:x_i^\infty)$ is again a congruence.

\begin{lem}\label{lemma: saturation is cong}
Let $C$ be a homogeneous congruence on $S$ and $x_0$ be one of the homogeneous coordinates. Then $(C:x_0^\infty)$ is a homogeneous congruence on $S$. 
\end{lem}
\begin{proof}
Reflexivity and symmetry are immediate. If $(F,G),(G,H)\in (C:x_0^\infty)$, choose $N,M\geq0$ such that 
\begin{equation}\label{eq: inside}
 (x_0^NF,x_0^NG), (x_0^MG,x_0^MH)\in C.   
\end{equation}
Multiplying the first relation by $x_0^M$ and the second by $x_0^N$ gives 
\begin{equation}
(x_0^{N+M}F,x_0^{N+M}G)\in C, \quad \text{and} \quad (x_0^{N+M}G,x_0^{N+M}H)\in C,
\end{equation}
so transitivity of $C$ gives $(x_0^{N+M}F,x_0^{N+M}H)\in C$ and hence $(F,H)\in (C:x_0^\infty)$. 

Compatibility with addition and multiplication is similar: if $(F,G),(F',G')\in (C:x_0^\infty)$, then for some $N$ and $M$, we have
\[
(x_0^NF,x_0^NG), (x_0^MF',x_0^MG')\in C. 
\]
Since $C$ is a congruence, it follows that
\[
(x_0^{N+M}FF',x_0^{N+M}GG') \in C,
\]
implying that $(FF',GG') \in (C:x_i^\infty)$. 

For addition, since $(x_0^K,x_0^K) \in C$ for any $K \in \mathbb{N}$, we have
\[
(x_0^{N+M}F,x_0^{N+M}G), (x_0^{N+M}F',x_0^{N+M}G')\in C, 
\]
implying that
\begin{equation}
 (x_0^{N+M}(F+F'),x_0^{N+M}(G+G'))\in C.  
\end{equation}
Hence $(F+F',G+G')\in (C:x_0^\infty)$.

Finally, we show that saturation preserves homogeneity. Suppose $(F,G)\in (C:x_0^\infty)$ and write $F=\sum_dF_d$ and $G=\sum_dG_d$ into homogeneous components. For some $N\geq0$ one has $(x_0^NF,x_0^NG)\in C$. Since $C$ is homogeneous, by Theorem \ref{theorem: congruence homogeneous elts}, $(x_0^NF_d,x_0^NG_d)\in C$ for every $d$, and therefore $(F_d,G_d)\in (C:x_0^\infty)$ for every $d$. Hence $(C:x_0^\infty)$ is homogeneous congruence on $S$.
\end{proof}

Among the homogeneous congruences on $\mathbb{B}[x_0,\dots,x_n]$, we will be particularly interested in the following case.

\begin{definition} \label{definition: bendh}
Let
$A=\mathbb B[x_1,\ldots,x_n], \ S=\mathbb B[x_0,\ldots,x_n],$
and let $P\subseteq A$ be an ideal. Define
\[
\operatorname{Bend}(P)^h
:=
\left\langle
(F\sim G)^h \mid 
(F,G)\in\operatorname{Bend}(P)
\right\rangle,
\]
which is a homogeneous congruence on $S$ by viewing $F$ and $G$ in $S$. We set
\[
C_P^{\operatorname{proj}}
:=
\left(\operatorname{Bend}(P)^h:x_0^\infty\right).
\]
\end{definition}

Now, we define the projective vanishing locus $V_+$ of $\Proj$ in a chartwise way as follows.

\begin{mydef}\label{definition: homegenius loci}
Let $C$ be a homogeneous congruence on $S=\mathbb{B}[x_0,\dots,x_n]$. We define $V_+(C)\subseteq\mathbb \Proj$ by specifying its intersection with each chart $U_i$ as follows:
\begin{equation}
V_+(C)\cap U_i:=V(C_{(x_i)})=\{Q\in\operatorname{SpecCong}((S_{x_i})_0)\mid C_{(x_i)}\subseteq Q\},    
\end{equation}
where $C_{(x_i)}$ is the degree-zero localization of $C$ at $x_i$ as in Definition \ref{definition: degree-zero}.
\end{mydef}

\begin{rmk}
Definition \ref{definition: homegenius loci} is analogous to the ordinary identity $V_+(I)\cap D_+(f)\cong V((I_f)_0)$ for homogeneous ideals in graded rings. These chartwise loci are compatible on overlaps. Indeed, restricting $C_{(x_i)}$ to $U_{ij}$ amounts to further localizing at $x_j/x_i$, and gives
\[
C_{(x_ix_j)}:=C_{x_ix_j}\cap\bigl((S_{x_ix_j})_0\times(S_{x_ix_j})_0\bigr).
\]
Starting instead from the $j$th chart and localizing at $x_i/x_j$ gives the same congruence $C_{(x_ix_j)}$. Hence $V(C_{(x_i)})$ and $V(C_{(x_j)})$ agree on $U_{ij}$, so the affine closed loci glue to a well-defined subset $V_+(C)\subseteq \Proj$.    
\end{rmk}

Now, we define the projective closure of an ideal. 

\begin{mydef}
Let $P$ be an ideal of $\mathbb{B}[x_1,\dots,x_n]$ and $\operatorname{Bend}(P)^h$ be the homogeneous congruence on $\mathbb{B}[x_0,\dots,x_n]$ as in Definition \ref{definition: bendh}. The \emph{projective closure} $\text{Proj}(P)$ of $P$ is defined by 
\[
\operatorname{Proj}(P):=V_+(C_P^{\operatorname{proj}})\subseteq \Proj.
\]
\end{mydef}

The following lemma shows that the previous definitions are compatible with the affine charts. For a homomorphism of semirings $\varphi:R_1 \to R_2$ and a congruence $C$ on $R_1$, by $\varphi(C)$, we mean the following set:
\[
\varphi(C):=\{(\varphi(a),\varphi(b) \in R_2 \times R_2 \mid (a,b) \in C\}.
\]

Let $S=\mathbb B[x_0,\ldots,x_n]$ with the standard grading, and let $P\subseteq A$ be an ideal. Let
\begin{equation}\label{eq: dehomo iso}
\theta:(S_{x_0})_0\longrightarrow A=\mathbb{B}[x_1,\dots,x_n], \quad \frac{x_i}{x_0} \mapsto x_i
\end{equation}
be the dehomogenization isomorphism.

\begin{lem}\label{lemma: main lemma}
With the same notation as above, we have
\[
\theta\bigl((C_P^{\operatorname{proj}})_{(x_0)}\bigr)=\operatorname{Bend}(P).
\]
Consequently, under the identification $(S_{x_0})_0\cong A$ and $U_0\cong\operatorname{SpecCong}(A)$, 
\[
\operatorname{Proj}(P)\cap U_0=V(\operatorname{Bend}(P)).
\]
\end{lem}
\begin{proof}
First, let $F\sim G$ be a relation in $\operatorname{Bend}(P)$. By definition of congruence homogenization, its homogenization $(F\sim G)^h$ lies in $\operatorname{Bend}(P)^h\subseteq C_P^{\operatorname{proj}}$. After localizing at $x_0$ and passing to degree zero, dehomogenization sends $(F\sim G)^h$ back to $F\sim G$. Indeed, if $D=\max\{\deg(F),\deg(G)\}$, then the homogenized relation is
\begin{equation}
 x_0^{D-\deg(F)}F^h\sim x_0^{D-\deg(G)}G^h.   
\end{equation}
After multiplying both sides by $x_0^{-D}$ in $S_{x_0}$ one obtains a degree-zero relation whose image under $\theta$ is exactly $F\sim G$. Hence
$\operatorname{Bend}(P)\subseteq\theta\bigl((C_P^{\operatorname{proj}})_{(x_0)}\bigr)$.

Conversely, let $\alpha\sim\beta$ be a relation in
$(C_P^{\operatorname{proj}})_{(x_0)}$. Since $\alpha$ and $\beta$ have
degree zero, after choosing a common denominator we may write
\[
\alpha=\frac{F}{x_0^d}
\qquad\text{and}\qquad
\beta=\frac{G}{x_0^d},
\]
where $F$ and $G$ are homogeneous of degree $d$. By Theorem~\ref{theorem: Tanaka} (4), there exists $m\geq0$
such that
$x_0^mF\sim x_0^mG$
lies in $C_P^{\operatorname{proj}}$.

Since
$C_P^{\operatorname{proj}} = \bigl(\operatorname{Bend}(P)^h:x_0^\infty\bigr),$
there exists $N\geq0$ such that
$x_0^{N+m}F\sim x_0^{N+m}G$
lies in $\operatorname{Bend}(P)^h$.

Let
$\delta:S=\mathbb{B}[x_0,\dots,x_n]\longrightarrow A=\mathbb{B}[x_1,\dots,x_n]$
be the dehomogenization homomorphism defined by
$\delta(x_0)=1$ and $\delta(x_i)=x_i$ for $1\leq i\leq n$. The inverse
image
$\delta^{-1}(\operatorname{Bend}(P))$
is a congruence on $S$. Moreover, it contains every homogenized
relation $(H\sim K)^h$ with $H\sim K$ in $\operatorname{Bend}(P)$,
because dehomogenization sends $(H\sim K)^h$ back to $H\sim K$.
Therefore
\[
\operatorname{Bend}(P)^h
\subseteq
\delta^{-1}(\operatorname{Bend}(P)).
\]
Applying $\delta$ to the relation
$x_0^{N+m}F\sim x_0^{N+m}G$ and using $\delta(x_0)=1$, we obtain
$\delta(F)\sim\delta(G)$
in $\operatorname{Bend}(P)$. Since
$\theta(\alpha)=\delta(F) \ \text{and}\  \theta(\beta)=\delta(G),$
it follows that
\[
\theta\bigl((C_P^{\operatorname{proj}})_{(x_0)}\bigr)
\subseteq
\operatorname{Bend}(P).
\]

Thus
$\theta\bigl((C_P^{\operatorname{proj}})_{(x_0)}\bigr)=\operatorname{Bend}(P)$.
By the chartwise definition of $V_+$, we have
\[
\operatorname{Proj}(P)\cap U_0
=V_+(C_P^{\operatorname{proj}})\cap U_0
=V\bigl((C_P^{\operatorname{proj}})_{(x_0)}\bigr).
\]
Under $U_0\cong\operatorname{SpecCong}(A)$, the equality above identifies this closed locus with $V(\operatorname{Bend}(P))$. Therefore
$\operatorname{Proj}(P)\cap U_0=V(\operatorname{Bend}(P))$.
\end{proof}

\begin{proposition}\label{proposition: closed proposition}
Let $C$ be a homogeneous congruence on
$S=\mathbb B[x_0,\ldots,x_n]$. Then $V_+(C)$ is closed in
$\mathbb P^n_{\mathbb B}$. In particular, for every ideal
$P\subseteq\mathbb B[x_1,\ldots,x_n]$,
$\operatorname{Proj}(P)$ is closed in $\mathbb P^n_{\mathbb B}$.
\end{proposition}

\begin{proof}
Let
\[
q:\coprod_{i=0}^n U_i\longrightarrow\mathbb P^n
\]
be the quotient map. By Definition \ref{definition: homegenius loci}, we have
\[
q^{-1}(V_+(C))
=
\coprod_{i=0}^n V(C_{(x_i)}).
\]
Each $V(C_{(x_i)})$ is closed in $U_i$, so this disjoint union is
closed in $\coprod_iU_i$. Since $\Proj$ is equipped with the
quotient topology, $V_+(C)$ is closed in $\Proj$. Taking
$C=C_P^{\operatorname{proj}}$ proves the final assertion.
\end{proof}

The following proposition shows that the quotient of $\mathbb{B}[x_0,\dots,x_n]$ by a homogeneous congruence has a canonical grading.

\begin{proposition}\label{proposition: homogeneous coordinate semiring}
Let $C$ be a homogeneous congruence on
$S=\mathbb B[x_0,\ldots,x_n]$, and set $R:=S/C$. Then the standard
grading on $S$ descends to a grading on $R$. Moreover, for each $i$
such that no power of $\overline{x_i}$ is zero in $R$, there are
natural isomorphisms of graded semirings:
\[
R_{\overline{x_i}}
\cong
S_{x_i}/C_{x_i} \quad \text{and} \quad \left(R_{\overline{x_i}}\right)_0
\cong
\frac{(S_{x_i})_0}{C_{(x_i)}}.
\]
\end{proposition}
\begin{proof}
We first show that the standard grading on $S$ descents to $R$. In fact, let $\pi:S \to R=S/C$ be the quotient map, and $R_d:=\pi(S_d)$. We claim that $R=\bigoplus_d R_d$. In fact, since $\pi$ is surjective, every element of $R$ is a sum of elements in $R_d$.  Suppose that
\begin{equation}
\sum_d\overline{F_d} = \sum_d\overline{G_d}.    
\end{equation}
Then, we have
\begin{equation}
\sum_dF_d\sim_C\sum_dG_d.    
\end{equation}
Since $C$ is homogeneous, by Theorem \ref{theorem: congruence homogeneous elts}, for any $d$, we have $F_d\sim_CG_d$, thereby showing that $\overline{F_d}=\overline{G_d}$. This proves our claim. Since multiplication satisfies
$R_dR_e\subseteq R_{d+e}$, this defines a grading on $R$.

Now consider the natural homomorphism
\[
\pi_i:S_{x_i}\longrightarrow R_{\overline{x_i}},
\qquad
\frac{F}{x_i^m}\longmapsto
\frac{\overline F}{\overline{x_i}^{\,m}}.
\]
Clearly $\pi_i$ is surjective. We claim that the congruence kernel of $\pi_i$ is exactly the localized
congruence $C_{x_i}$. Indeed, two fractions have the same image precisely when, after multiplying by a sufficiently large power of $x_i$, their cross-products are equivalent modulo $C$, which is the definition of equivalence modulo $C_{x_i}$. It therefore induces an isomorphism
\begin{equation}
 S_{x_i}/C_{x_i} \cong R_{\overline{x_i}}.   
\end{equation}

The localized congruence $C_{x_i}$ is homogeneous for the induced
$\mathbb Z$-grading on $S_{x_i}$, so the quotient
$S_{x_i}/C_{x_i}$ inherits a grading whose degree-zero component is
the image of $(S_{x_i})_0$. Consider the natural homomorphism\footnote{Note that $[\alpha]_{C_{(x_i)}}$ is the equivalence class of $\alpha \in (S_{x_i})_0$ in the quotient $\frac{(S_{x_i})_0}{C_{(x_i)}}$ and $[\alpha]_{C_{x_i}}$ is similarly read in $\left(S_{x_i}/C_{x_i}\right)_0$.}
\[
\eta:
\frac{(S_{x_i})_0}{C_{(x_i)}}
\longrightarrow
\left(S_{x_i}/C_{x_i}\right)_0,
\qquad
[\alpha]_{C_{(x_i)}}\longmapsto[\alpha]_{C_{x_i}}.
\]
This map is well defined because
$C_{(x_i)} = C_{x_i}\cap
\bigl((S_{x_i})_0\times(S_{x_i})_0\bigr).$
It is surjective because, by definition of the quotient grading,
every element of $\left(S_{x_i}/C_{x_i}\right)_0$ is represented by
an element of $(S_{x_i})_0$. It is injective because, for
$\alpha,\beta\in(S_{x_i})_0$, equality
$[\alpha]_{C_{x_i}}=[\beta]_{C_{x_i}}$
is equivalent to $(\alpha,\beta)\in C_{x_i}$, which, since both
elements have degree zero, is equivalent to
$(\alpha,\beta)\in C_{(x_i)}$. Hence $\eta$ is an isomorphism.

Combining this with the preceding isomorphism gives
\[
\left(R_{\overline{x_i}}\right)_0
\cong \left(S_{x_i}/C_{x_i}\right)_0 \cong
\frac{(S_{x_i})_0}{C_{(x_i)}}.
\]
\end{proof}

\begin{rmk}
We remark that from Theorem \ref{theorem: Tanaka} (1), for a homogeneous congruence $C$ on $S=\mathbb B[x_0,\ldots,x_n]$, we have
\[
V_+(C)\cap U_i\cong\operatorname{SpecCong}\!\left(\frac{(S_{x_i})_0}{C_{(x_i)}}\right),
\]
where $U_i$ is the standard chart as in Definition \ref{definition: U_i}.
\end{rmk}

Here are several direct consequences. Let $P$ be an ideal of $S=\mathbb{B}[x_0,\dots,x_n]$. Define the following:
\begin{equation}\label{eq: chartwise coordiate}
R_P^{\operatorname{proj}} = S/C_P^{\operatorname{proj}}, \quad R_{P,i} = \frac{(S_{x_i})_0}
{(C_P^{\operatorname{proj}})_{(x_i)}} \quad \text {for each $i$}.
\end{equation}
$R_P^{\operatorname{proj}}$ is a graded semiring by Proposition \ref{proposition: homogeneous coordinate semiring}, which we call the \emph{homogeneous coordinate semiring} associated to the projective bend
locus $\text{Proj}(P)$. Again, Proposition \ref{proposition: homogeneous coordinate semiring} shows that 
\begin{equation}
\operatorname{Proj}(P)\cap U_i
\cong
\operatorname{SpecCong}(R_{P,i}).
\end{equation}
Moreover, if no power of $\overline{x_i}$ is zero in
$R_P^{\operatorname{proj}}$, then
\begin{equation}
R_{P,i}
\cong
\left(
(R_P^{\operatorname{proj}})_{\overline{x_i}}
\right)_0.
\end{equation}

Finally, by Lemma~\ref{lemma: main lemma}, on the affine chart $U_0$ we have
$\theta\bigl((C_P^{\operatorname{proj}})_{(x_0)}\bigr) = \operatorname{Bend}(P).$
Since $\theta:(S_{x_0})_0\to\mathbb B[x_1,\ldots,x_n]$ is an
isomorphism, this induces an isomorphism
\begin{equation}
R_{P,0}
=
\frac{(S_{x_0})_0}
{(C_P^{\operatorname{proj}})_{(x_0)}}
\cong
\frac{\mathbb B[x_1,\ldots,x_n]}
{\operatorname{Bend}(P)}.
\end{equation}

\begin{rmk}
The chartwise quotient semirings $R_{P,i}$ in \eqref{eq: chartwise coordiate} retain the full defining relations, whereas the underlying closed point set remembers only the prime congruences containing those relations. We retain this algebraic data (coordinate semirings) without asserting here that the prime-congruence space carries a structure sheaf.
\end{rmk}

\subsection{Comparison with other works.}

In this subsection, we discuss some relations between our approach, that of Jo\'o and Mincheva \cite{joo2026varieties,joo2018prime}, and that of Lescot \cite{les3}.

The homogenization $C\mapsto C^h$ used in this paper is the standard homogenization of congruences appearing in the theory of tropical schemes by Maclagan and Rinc\'on in \cite{maclagan2016tropical}. Our construction of $\Proj$ and $\text{Proj}(P)$ differ from the usual tropical scheme construction as we consider prime congruences on the affine chart semirings and glued chartwise.

The affine relationship between ideals and congruences studied by Jo\'o and Mincheva \cite[\S 3.2]{joo2026varieties} is closely related to the first step of our construction. For a congruence $C$ on $A=\mathbb{B}[x_1,\dots,x_n]$, set
\begin{equation}
I_C
:=
\left\{
f \mid 
(f,f_{\widehat u})\in C
\text{ for every }u\in\operatorname{Supp}(f)
\right\},
\end{equation}
and for an ideal $I$ of $A$ define its bend closure by
$I^{\operatorname{cl}}:=I_{\operatorname{Bend}(I)}.$
Then, one has 
\begin{equation}
I\subseteq I^{\operatorname{cl}} \quad \text{and} \quad 
\operatorname{Bend}(I^{\operatorname{cl}})=\operatorname{Bend}(I).
\end{equation}
Consequently, one has
\begin{equation}
C_I^{\operatorname{proj}}=C_{I^{\operatorname{cl}}}^{\operatorname{proj}}
\quad \text{and}\quad
\operatorname{Proj}(I)
=
\operatorname{Proj}(I^{\operatorname{cl}}).
\end{equation}
Thus our projective closure construction depends only on the affine bend congruence. In the tropical setting with $\mathbb{B}[x_1^\pm,\dots,x_n^\pm]$, Jo\'o and Mincheva
\cite[Theorem~3.11(iii)]{joo2026varieties} showed that the correspondences
\begin{equation}
I\mapsto\operatorname{Bend}(I)\quad  \text{and} \quad C\mapsto I_C   
\end{equation}
 restrict to a
bijection between nonzero closed prime ideals\footnote{An ideal $I$ is closed if $I=I_{\text{Bend}(I)}$.} and non-minimal prime
congruences. This result is affine and does not imply that
$\operatorname{Bend}(I)$ is a prime congruence for every prime ideal
$I$ in a Boolean polynomial semiring, nor that
\[
C_I^{\operatorname{proj}}
=
\left(\operatorname{Bend}(I)^h:x_0^\infty\right)
\]
is a prime congruence on
$S=\mathbb B[x_0,\ldots,x_n]$.

Jo\'o and Mincheva also showed why one should retain the full prime-congruence spectrum rather than only ordinary tropical-valued points (geometric points): a prime congruence may contain ordering data that are invisible at the level of geometric points. Accordingly, the local pieces of our projective closure are the full spectra
$V(C_{(x_i)})\subseteq\operatorname{SpecCong}((S_{x_i})_0),$
not merely their sets of $\mathbb B$- or $\mathbb T$-valued points.

Lescot  \cite[\S 2]{les3} studied a different notion of prime congruence, namely a proper congruence $Q$ satisfying
\[
ab\sim_Q0\Longrightarrow a\sim_Q0\ \text{or}\ b\sim_Q0.
\]
The twisted product primality implies Lescot's condition because $(a,0)(b,0)=(ab,0)$, but the converse need not hold. Moreover, Lescot's topology is determined by zero classes, while the topology used here records arbitrary relations contained in a prime congruence. His saturated spectrum is therefore related to, but different from, the prime-congruence spectrum used in our construction.

Finally, the classifications of prime congruences of Boolean Laurent
polynomial and polynomial semirings by admissible matrices in
\cite[Theorems~4.6 and~4.9]{joo2018prime} may be applied separately
on each affine chart and its standard overlaps. Such matrices describe the prime congruences containing $C_{(x_i)}$, while their compatibility on overlaps is obtained by the usual change of affine coordinates $x_k/x_i\mapsto(x_k/x_j)/(x_i/x_j)$. The resulting descriptions remain local; the projective locus is obtained by gluing the compatible affine pieces.

\begin{rmk}
By Proposition~\ref{prop:Tanaka-irreducibility}, the use of a prime
ideal $P$ as affine input does not by itself imply that
$V(\operatorname{Bend}(P))$ is irreducible. Assuming this locus is
nonempty, it is irreducible precisely when
\[
\sqrt[\operatorname{pr}]{\operatorname{Bend}(P)}
\]
is a prime congruence. We return to the corresponding projective
statement after proving the projective-closure theorem (Theorem \ref{thm:projective-congruence-closure} and Corollary \ref{cor:projective-irreducibility}). 
\end{rmk}

\subsection{Example: Projective cuspidal cubic}\label{example: plane curve}  In this subsection, we compute one example explicitly. We start with the following lemma. 

\begin{lem}
Let $C$ be a proper congruence on
$S=\mathbb B[x_0,\ldots,x_n]$, and assume that no power of $x_i$ is
$C$-equivalent to zero. If $\overline{x_i}$ denotes the image of
$x_i$ in $S/C$, then
$(C:x_i^\infty)=C$
if and only if the natural localization map
$S/C\longrightarrow (S/C)_{\overline{x_i}}$
is injective.    
\end{lem}
\begin{proof}
For $F,G\in S$, the images of $\overline F$ and $\overline G$
in $(S/C)_{\overline{x_i}}$ are equal if and only if there exists
$N\geq0$ such that
$\overline{x_i}^{\,N}\overline F = \overline{x_i}^{\,N}\overline G$
in $S/C$, or equivalently,
$x_i^NF\sim_Cx_i^NG.$
Thus the kernel congruence of the localization map is induced by
$(C:x_i^\infty)$, and the map is injective precisely when
$(C:x_i^\infty)=C$.   
\end{proof}

Consequently, whenever $C_P^{\operatorname{proj}}$ is proper, the
natural map
\[
S/C_P^{\operatorname{proj}}
\longrightarrow
\left(S/C_P^{\operatorname{proj}}\right)_{\overline{x_0}}
\]
is injective. Equivalently, if two elements $F,G\in S$ become equal
modulo $C_P^{\operatorname{proj}}$ after localizing at $x_0$, then
$F\sim_{C_P^{\operatorname{proj}}}G$ already in $S$. In particular,
restriction to the distinguished chart $U_0$ introduces no additional
homogeneous relations as it is shown in the example below. \\

Let $A=\mathbb B[u,v]$, let $S=\mathbb B[x_0,x_1,x_2]$, and identify the standard affine chart $U_0$ with $\operatorname{SpecCong}(\mathbb B[u,v])$ by $u=x_1/x_0$ and $v=x_2/x_0$. Consider the ideal
\[
P=\langle u^3+v^2\rangle\subseteq\mathbb B[u,v].
\]

We first claim that $\Bend(P)=\angles{u^3 \sim v^2}$. In fact, since $u^3+v^2\in P$, its bend relations give
\[
u^3+v^2\sim u^3,\quad u^3+v^2\sim v^2,
\]
and hence $u^3\sim v^2$. Thus $\langle u^3\sim v^2\rangle\subseteq\operatorname{Bend}(P)$.

Conversely, let $C:=\langle u^3\sim v^2\rangle$ and let $f\in P$. Since $P$ is principal, we may write $f=h(u,v)(u^3+v^2)$ for some $h\in\mathbb B[u,v]$. Let $q$ be a monomial occurring in $f$. Choose a monomial $m$ occurring in $h$ such that either $q=mu^3$ or $q=mv^2$, and let $q'$ denote the corresponding paired monomial $mv^2$ or $mu^3$, respectively. Then, we have
\[
(q,q') \in C.
\]
Moreover, $q'\neq q$ and $q'$ also occurs in $f$, so $q'$ remains in $f_{\widehat q}$. Since addition is idempotent, we have
\[
f=f_{\widehat q}+q\sim f_{\widehat q}+q'=f_{\widehat q}.
\]
Thus every bend relation of every element of $P$ lies in $C$, and therefore $\operatorname{Bend}(P)=\langle u^3\sim v^2\rangle.$

The affine coordinate semiring can be described explicitly. Consider the homomorphism 
\begin{equation}
  \varphi:\mathbb B[u,v]\longrightarrow\mathbb B[t],\quad \varphi(u)=t^2,\ \varphi(v)=t^3.  
\end{equation}
The relation $u^3\sim v^2$ lies in $\ker(\varphi)$.\footnote{By $\ker(\varphi)$, we mean the congruence kernel, i.e., equalizer.} Conversely, define
$\lambda:\mathbb N^2\longrightarrow\mathbb N, \ \lambda(a,b)=2a+3b.$
Two monomials $u^av^b$ and $u^{a'}v^{b'}$ have the same image under
$\varphi$ precisely when
\[
\lambda(a,b)=2a+3b=2a'+3b'=\lambda(a',b'),
\]
or equivalently,
$(a-a',b-b')=k(3,-2)$
for some $k\in\mathbb Z$. Since both exponent vectors have
nonnegative coordinates, one monomial can therefore be obtained from
the other by repeated applications of the relation
$u^3\sim v^2$. Hence any two monomials with the same image under
$\varphi$ are equivalent modulo
$C=\langle u^3\sim v^2\rangle$.

We now verify the corresponding statement for arbitrary polynomials.
Suppose that $\varphi(f)=\varphi(g)$, and let
$E:=\operatorname{Supp}(\varphi(f))
=\operatorname{Supp}(\varphi(g)).$
For each $\ell\in E$, choose one monomial $m_\ell$ satisfying
$\varphi(m_\ell)=t^\ell$. Every monomial occurring in either $f$ or
$g$ is equivalent modulo $C$ to the corresponding $m_\ell$.
Consequently, using compatibility with addition and the idempotence of
addition, we obtain
\[
f\sim_C\sum_{\ell\in E}m_\ell
\sim_C g.
\]
Thus $\ker(\varphi)\subseteq C$. The reverse inclusion follows from
$\varphi(u^3)=\varphi(v^2)$, and therefore
$\ker(\varphi)=\langle u^3\sim v^2\rangle
=\operatorname{Bend}(P).$
Hence
\[
\mathbb B[u,v]/\operatorname{Bend}(P)
\cong
\mathbb B[t^2,t^3]\subseteq\mathbb B[t].
\]

We now projectivize. With, $u=x_1/x_0$ and $v=x_2/x_0$, the relation $u^3\sim v^2$ homogenizes to
\begin{equation}
x_1^3\sim x_0x_2^2.
\end{equation}
Now, consider the following congruence on $\mathbb{B}[x_0,x_1,x_2]$:
\begin{equation}
H:=\langle x_1^3\sim x_0x_2^2\rangle.
\end{equation}
To describe the congruence $H$, consider the homomorphism 
\begin{equation}
 \psi:S=\mathbb{B}[x_0,x_1,x_2]\longrightarrow\mathbb B[s,t],\quad \psi(x_0)=s^3,\ \psi(x_1)=st^2,\ \psi(x_2)=t^3.   
\end{equation}

We claim that $H=\ker(\psi)$. Clearly,
$x_1^3\sim x_0x_2^2$ lies in $\ker(\psi)$, so
$H\subseteq\ker(\psi).$
Define the monoid homomorphism
\[
\mu:\mathbb N^3\longrightarrow\mathbb N^2,
\qquad
\mu(a,b,c)=(3a+b,2b+3c).
\]
Then
$\psi(x_0^ax_1^bx_2^c) = s^{3a+b}t^{2b+3c},$
so two monomials $x_0^ax_1^bx_2^c$ and
$x_0^{a'}x_1^{b'}x_2^{c'}$ have the same image under $\psi$ if and
only if
$\mu(a,b,c)=\mu(a',b',c').$
The induced group homomorphism
\[
\mu^{\operatorname{gp}}:\mathbb Z^3\longrightarrow\mathbb Z^2
\]
has kernel generated by $(1,-3,2)$. Hence the two monomials have the
same image precisely when
$(a-a',b-b',c-c')=k(1,-3,2)$
for some $k\in\mathbb Z$. This is exactly the exponent relation
corresponding to
$x_1^3\sim x_0x_2^2.$
Therefore any two monomials with the same image under $\psi$ are
equivalent modulo $H$.

Now suppose that $\psi(F)=\psi(G)$, and let
\[
E:=\operatorname{Supp}(\psi(F))
=\operatorname{Supp}(\psi(G)) \subseteq \mathbb{N}^2.
\]
For each $(p,q)\in E$, choose one monomial $M_{p,q}$ satisfying
$\psi(M_{p,q})=s^pt^q.$
Every monomial occurring in either $F$ or $G$ is equivalent modulo
$H$ to the corresponding $M_{p,q}$. Hence
\[
F\sim_H\sum_{(p,q)\in E}M_{p,q}
\sim_H G.
\]
It follows that $\ker(\psi)\subseteq H$, and therefore
$H=\ker(\psi).$

Since $H=\ker(\psi)$, multiplication by $x_0$ is cancellative modulo
$H$. Indeed, if
\[
(x_0^NF,x_0^NG)\in H,
\text{ then }
s^{3N}\psi(F)=s^{3N}\psi(G).
\]
Multiplication by the monomial $s^{3N}$ is injective on
$\mathbb B[s,t]$, since it translates monomial supports. Therefore
$\psi(F)=\psi(G)$, and hence $(F,G)\in H$. Thus
$(H:x_0^\infty)=H.$

It remains to verify that homogenizing the entire congruence
$\operatorname{Bend}(P)$ gives $H$, rather than merely observing that
the displayed generator homogenizes to the generator of $H$. Since
$u^3\sim v^2$ lies in $\operatorname{Bend}(P)$, its homogenization
$x_1^3\sim x_0x_2^2$ lies in $\operatorname{Bend}(P)^h$, and hence
$H\subseteq\operatorname{Bend}(P)^h.$

Conversely, let $F\sim G$ be any relation in
$\operatorname{Bend}(P)=\ker(\varphi)$, and let
$D=\max\{\deg(F),\deg(G)\}.$
For a monomial $u^av^b$, its degree-$D$ homogenization is
$x_0^{D-a-b}x_1^ax_2^b$, whose image under $\psi$ is
\[
s^{3(D-a-b)}(st^2)^a(t^3)^b
=
s^{3D-(2a+3b)}t^{2a+3b}.
\]
Thus the image under $\psi$ of a homogenized monomial depends only on
$2a+3b$. Since $\varphi(F)=\varphi(G)$, the two polynomials have the
same set of values of $2a+3b$ in their supports. It follows that the
two sides of $(F\sim G)^h$ have the same image under $\psi$, and hence
\[
(F\sim G)^h\in\ker(\psi)=H.
\]
Therefore
$\operatorname{Bend}(P)^h\subseteq H.$
and we conclude that
$\operatorname{Bend}(P)^h=H.$

Since $H$ is already $x_0$-saturated,
\[
C_P^{\operatorname{proj}}
=
(\operatorname{Bend}(P)^h:x_0^\infty)
=
H
=
\langle x_1^3\sim x_0x_2^2\rangle.
\]
We compute the resulting projective locus chartwise. On the chart 
\[
U_0\cong\operatorname{SpecCong}(\mathbb B[u,v]),\ u=x_1/x_0,\ v=x_2/x_0,
\]
the homogeneous relation becomes $u^3\sim v^2.$ Hence, we have
\[
\operatorname{Proj}(P)\cap U_0 =V(\langle u^3\sim v^2\rangle),
\]
whose coordinate semiring is
\[
\mathbb B[u,v]/\langle u^3\sim v^2\rangle\cong\mathbb B[t^2,t^3].
\]
On the chart 
\[
U_1\cong\operatorname{SpecCong}(\mathbb B[a,b]),\ a=x_0/x_1,\ b=x_2/x_1,
\]
dividing $x_1^3\sim x_0x_2^2$ by $x_1^3$ gives $1\sim ab^2.$ Thus, one has
\[
\operatorname{Proj}(P)\cap U_1 =V(\langle1\sim ab^2\rangle).
\]
In the quotient, $b$ is invertible with inverse $ab$, and $a\sim b^{-2}$, so the coordinate semiring is naturally identified with the Laurent Boolean semiring $\mathbb B[b^{\pm1}]$.
On the chart 
\[
U_2\cong\operatorname{SpecCong}(\mathbb B[c,d]),\ c=x_0/x_2,\ d=x_1/x_2,
\]
the relation becomes $d^3\sim c.$ Therefore, we have
\[
\operatorname{Proj}(P)\cap U_2 =V(\langle d^3\sim c\rangle),
\]
and the coordinate semiring on this chart is $\mathbb B[c,d]/\langle d^3\sim c\rangle\cong\mathbb B[d].$

These chartwise descriptions agree on the standard overlaps. For example, on $U_0\cap U_1$ one has $a=u^{-1}$ and $b=v/u$, so by $u^3\sim v^2$
\[
ab^2=u^{-1}(v/u)^2=v^2u^{-3}\sim 1.
\]
Similarly, on $U_0\cap U_2$ one has $c=v^{-1}$ and $d=u/v$, and hence
\[
d^3=u^3v^{-3}\sim v^2v^{-3}=v^{-1}=c.
\]
Thus the affine congruence loci glue to the projective closed locus 
\[
\operatorname{Proj}(P) =V_+(\langle x_1^3\sim x_0x_2^2\rangle)\subseteq\Proj.
\]
This is a projective plane curve in our setting: its defining homogeneous relation is cubic, and its distinguished affine chart is the monomial curve $u^3\sim v^2$ with coordinate semiring $\mathbb B[t^2,t^3]$. This is the analogue of how over a field the ideal $(u^3-v^2)\subseteq k[u,v]$ defines the curve $u^3=v^2$, and its projective closure is cut out by
$x_1^3-x_0x_2^2=0$
in $\mathbb P^2_k$.\\

By a $\mathbb B$-valued point of an affine chart with coordinate semiring $R$ we mean a prime congruence $Q$ such that $R/Q\cong\mathbb B$ as $\mathbb B$-algebras. Equivalently, $Q$ is the congruence kernel of a $\mathbb B$-algebra homomorphism
$R\longrightarrow\mathbb B.$ In particular, this agrees with the notion of $K$-rational points of an algebraic variety over a field $k$ with an extension $K$ of $k$. 

We emphasize that the full prime-congruence locus need not coincide
with its subset of $\mathbb B$-valued points as in the classical case for fields. For the projective locus considered, these chartwise evaluation
points correspond to the nonzero tuples
\[
[a_0:a_1:a_2]\in\mathbb B^3\setminus\{(0,0,0)\}
\]
satisfying the defining homogeneous relation
$a_1^3=a_0a_2^2.$
Since $\mathbb B^\times=\{1\}$, there is no nontrivial rescaling of
such tuples. 

Indeed, we claim that the projective curve
\[
X:=V_+\bigl(\langle x_1^3\sim x_0x_2^2\rangle\bigr)
\subseteq \Proj
\]
has five prime-congruence points, although only three of them are
$\mathbb B$-valued. To see this, we continue using the same notation as above.

On the chart $U_0$, with $u=x_1/x_0$ and $v=x_2/x_0$, the coordinate
semiring is
\begin{equation}\label{eq: coordinate semiring0}
\mathbb B[u,v]/\langle u^3\sim v^2\rangle
\cong
\mathbb B[t^2,t^3].
\end{equation}
Consider the following (additive) submonoid of $\mathbb{N}$
\begin{equation}\label{eq: M}
M:=\langle2,3\rangle=\{0,2,3,\ldots\},
\end{equation}
equipped with the trivial order. Then the semiring in
\eqref{eq: coordinate semiring0} is the monoid semiring $\mathbb B[M]$.

Recall from \cite[Definition~3.6]{Tanaka26} that a \emph{face} of an ordered monoid $M$ is a submonoid $F\subseteq M$ such that
\[
a+b\in F\Longrightarrow a,b\in F,
\]
and such that $a\in F$ and $a\leq b$ imply $b\in F$. Since $M$ is
equipped here with the trivial order, the second condition is
automatic.

Because $M\subseteq\mathbb N$, we write its elements multiplicatively
as monomials in one formal variable:
\[
\mathbb B[M]
=
\left\{
\sum_{m\in E}t^m:
E\subseteq M\text{ finite}
\right\},
\qquad
t^mt^{m'}=t^{m+m'}.
\]
Thus $t^m$ is the single-variable monomial indexed by $m\in M$.

For a prime congruence $Q$ on $\mathbb B[M]$, define 
\[
F_Q:=\{m\in M:t^m\not\sim_Q0\}.
\]
Tanaka shows that $F_Q$ is a face of $M$. Moreover, for each face
$F$, the stratum consisting of prime congruences with 
$F = F_Q$ is naturally identified with
\[
\operatorname{SpecCong}\bigl(\mathbb B[F^{\operatorname{gp}}]\bigr),
\]
where $F^{\operatorname{gp}}$ is the group completion of $F$; see
\cite[Lemma~4.14 and equation~(23)]{Tanaka26}.

In our case of \eqref{eq: M}, the only faces of $M$ are $\{0\}$ and $M$. In fact, suppose that a face
$F$ contains a positive element. Expressing that element as a sum of
$2$'s and $3$'s shows that $F$ contains either $2$ or $3$. If
$2\in F$, then $6=2+2+2\in F$, and the equality $6=3+3$ implies
$3\in F$. Similarly, $3\in F$ implies $2\in F$. Hence $F=M$.

The face $\{0\}$ contributes one point, since
$\mathbb B[\{0\}^{\operatorname{gp}}]=\mathbb B$ has only its diagonal
prime congruence. On the chart $U_0$, this point is the prime
congruence induced by evaluation at $t=0$:
\[
Q_0^{(0)}
:=
\ker\!\left(
\mathbb B[t^2,t^3]\longrightarrow\mathbb B,
\quad
t^2,t^3\longmapsto0
\right).
\]
The superscript records that this point arises on $U_0$. For the face
$M$, one has
\[
M^{\operatorname{gp}}=\mathbb Z,
\qquad
\mathbb B[M^{\operatorname{gp}}]=\mathbb B[t^{\pm1}].
\]

Now, for $\mathbb B[t^{\pm1}]$, the one-variable case of \cite[Theorem~4.6]{joo2018prime} gives exactly three prime congruences on this Laurent semiring. Explicitly, because the image of $t$ is
invertible in a totally ordered quotient, exactly one of the following
holds:
\[
t\sim1,\qquad t>1,\qquad t<1.
\]
These give, respectively,
\[
Q_1
:=
\ker\!\left(
\mathbb B[t^{\pm1}]\longrightarrow\mathbb B,
\quad
t\longmapsto1
\right),
\]
which identifies all nonzero elements, the congruence
\[
f\sim_{Q_{\max}}g
\Longleftrightarrow
\max\operatorname{Supp}(f)=\max\operatorname{Supp}(g),
\]
and the congruence
\[
f\sim_{Q_{\min}}g
\Longleftrightarrow
\min\operatorname{Supp}(f)=\min\operatorname{Supp}(g).
\]
Consequently,
\[
\left|
\operatorname{SpecCong}\bigl(\mathbb B[t^2,t^3]\bigr)
\right|
=
1+3=4.
\]

On the chart $U_2$, set $c=x_0/x_2$ and $d=x_1/x_2$. The defining
relation becomes $d^3\sim c$, so the coordinate semiring is
\[
\mathbb B[c,d]/\langle d^3\sim c\rangle
\cong
\mathbb B[d].
\]
Applying the same mobile-face argument to
$\mathbb B[d]=\mathbb B[\mathbb N]$ shows that this chart also has four
prime-congruence points. Its boundary point is
\[
Q_0^{(2)}
:=
\ker\!\left(
\mathbb B[d]\longrightarrow\mathbb B,
\quad
d\longmapsto0
\right),
\]
while its other three points are identified on the Laurent overlap
with $Q_1,Q_{\max},Q_{\min}$.

The overlap $U_0\cap U_2$ is obtained by inverting $v=t^3$, and
\[
\mathbb B[t^2,t^3][(t^3)^{-1}]
=
\mathbb B[t^{\pm1}].
\]
It therefore contains exactly the three Laurent prime congruences
$Q_1,Q_{\max},Q_{\min}$. On the chart $U_1$, with
$a=x_0/x_1$ and $b=x_2/x_1$, the relation is
$1\sim ab^2,$
and its coordinate semiring is $\mathbb B[b^{\pm1}]$. Under the change
of coordinates
\[
b=\frac{x_2/x_0}{x_1/x_0}=\frac vu=t,
\]
this chart is identified with the same Laurent overlap. Thus $U_1$
contributes no additional points. Moreover, in the $U_1$-coordinate semiring the relation
$1\sim ab^2$ implies that both $a$ and $b$ are invertible. Hence
every point of $X\cap U_1$ lies in both $U_0$ and $U_2$. Since the standard charts cover $X$
and $X\cap U_1\subseteq X\cap U_0\cap U_2$, we obtain
\[
|X|=4+4-3=5.
\]

It remains to identify which of these five points are
$\mathbb B$-valued. A nonzero projective tuple
$[a_0:a_1:a_2]\in\mathbb B^3\setminus\{(0,0,0)\}$
belongs to $X$ precisely when
$a_1^3=a_0a_2^2.$
Since $a^m=a$ for every $a\in\mathbb B$ and every $m\geq1$, this is
equivalent to
$a_1=a_0a_2.$
There are therefore exactly three $\mathbb B$-valued points:
\[
[1:0:0],\qquad [0:0:1],\qquad [1:1:1].
\]

The point $[1:0:0]$ lies in $U_0$ and has affine coordinates
$u=\frac{x_1}{x_0}=0, \ v=\frac{x_2}{x_0}=0.$
It therefore corresponds to the boundary evaluation congruence
$Q_0^{(0)}$.
Similarly, $[0:0:1]$ lies in $U_2$ and has affine coordinates
$c=\frac{x_0}{x_2}=0, \ d=\frac{x_1}{x_2}=0.$
It corresponds to the boundary evaluation congruence $Q_0^{(2)}$.

On the chart $U_1$, a $\mathbb B$-valued point is determined by
$a,b\in\mathbb B$ satisfying
$1=ab^2.$
This forces $a=b=1$, giving the point $[1:1:1]$. Since both
coordinates are nonzero, this point lies in the Laurent overlap. Under the coordinate $t=b$, it corresponds to $Q_1$, induced by evaluation at $t=1$. Thus $U_1$ contains a $\mathbb B$-valued point but contributes no additional projective point.

Hence the three $\mathbb B$-valued points are
\[
Q_0^{(0)},\qquad Q_0^{(2)},\qquad Q_1,
\]
corresponding respectively to
$[1:0:0],\ [0:0:1],\ [1:1:1].$
The remaining two points, $Q_{\max}$ and $Q_{\min}$, are not
$\mathbb B$-valued, since their quotients retain the nontrivial
ordering information $t>1$ and $t<1$, respectively, and hence are not
isomorphic to $\mathbb B$. Thus the five points of $X$ are
\[
Q_0^{(0)},\qquad
Q_0^{(2)},\qquad
Q_1,\qquad
Q_{\max},\qquad
Q_{\min}.
\]
The prime-congruence locus therefore contains ordering data that are
invisible in the set of ordinary $\mathbb B$-valued solutions. The
coordinate semiring is likewise algebraic coordinate data rather than
the set of points itself, and may be infinite even when the
prime-congruence locus is finite.

\begin{rmk}
The preceding chartwise construction is motivated by the local description of $\operatorname{Proj}$ in classical algebraic geometry. The chartwise approach seems to be essential for prime congruences because the naive analogue of the global set-theoretic definition of $\operatorname{Proj}$ is degenerate. Let $S=\mathbb B[x_0,\ldots,x_n]$ with the standard grading, and suppose that $Q$ is a proper homogeneous prime congruence on $S$. By \cite[Proposition~2.14]{Tanaka26}, the quotient $S/Q$ is totally
ordered and cancellative. Hence for every positive-degree homogeneous polynomial $F$ one has either $1+F\sim_Q1$ or $1+F\sim_QF$. Homogeneity forces $F\sim_Q0$ in the first case and $1\sim_Q0$ in the second. The second alternative is impossible because $Q$ is proper, so every positive-degree homogeneous polynomial is equivalent to zero. Hence the irrelevant ideal $S_+$ is contained in the zero class of every proper homogeneous prime congruence. Therefore the naive global spectrum consisting of homogeneous prime congruences whose zero classes do not contain $S_+$ is empty, which necessitates the chartwise construction used above.
\end{rmk}

\begin{rmk}
Tanaka's results, collected in Theorem~\ref{theorem: Tanaka}, provide the affine topological framework underlying
the chartwise construction above. By
\cite[Proposition~2.8]{Tanaka26}, quotienting by a congruence
identifies the prime-congruence spectrum of the quotient with the
corresponding closed subset. Moreover,
\cite[Propositions~2.12 and~2.13]{Tanaka26} show that localization
identifies the spectrum of the localized semiring with an open subset
of the original spectrum and is compatible with quotienting. These
results justify gluing the affine prime-congruence spectra.

Tanaka constructs tropical toric spaces by gluing relative
prime-congruence spectra of affine toric $\mathbb T$-algebras
associated to the cones of a fan; see \cite[\S6.2]{Tanaka26}. In
particular, applying his construction to the standard fan of
projective space gives the corresponding prime-congruence toric
analogue of tropical projective space. The space $\Proj$ is the
Boolean analogue of this chartwise construction. As observed in the
preliminaries, every proper prime congruence on a $\mathbb B$-algebra
lies over $\mathbb B$, so no restriction of the full spectrum is
needed in the Boolean case.

We emphasize that Tanaka does not consider the projective-closure operation which we construct in this paper:
he does not define the homogenization and saturation
\[
C\longmapsto (C^h:x_0^\infty)
\]
or identify the resulting locus with the topological closure of an
affine congruence locus. His quotient, localization, and gluing
results provide the affine topological framework for the
construction, while the projective-closure procedure is developed in
this section.
\end{rmk}

\section{Topological closure}

In this section, we prove that the projective closure of an ideal $P \subseteq \mathbb{B}[x_1,\dots,x_n]$ is the topological closure of $V(\Bend(P))$ in $\Proj$. 

We will again use the following notation: $A=\mathbb B[t_1,\ldots,t_n],
\
S=\mathbb B[x_0,\ldots,x_n],$
and 
\begin{equation}\label{eq: dehomo}
\theta:(S_{x_0})_0\xrightarrow{\sim}A, \ \theta(x_i/x_0)=t_i \quad (\textrm{dehomogenization isomorphism)}.
\end{equation}

We begin by showing that there is a bijection between congruences on $A$ and the homogeneous $x_0$-saturated congruences on $S$. 

\begin{proposition}
\label{prop:congruence-hom-dehom}
Let $\operatorname{Cong}(A)$ denote the set of congruences on $A$, and let
$\operatorname{Cong}^h_{\operatorname{sat}}(S)$ denote the set of
homogeneous $x_0$-saturated congruences on $S$. The following function:
\begin{equation}
\varphi:\operatorname{Cong}(A)\longrightarrow
\operatorname{Cong}^h_{\operatorname{sat}}(S),
\qquad
\varphi(C)=(C^h:x_0^\infty),
\end{equation}
where $C^h=\{(F\sim G)^h \mid (F,G) \in C\}$ (viewed $C$ in $S \times S$), is an order-preserving bijection with the following inverse:
\begin{equation}
\psi:\operatorname{Cong}^h_{\operatorname{sat}}(S)
\longrightarrow\operatorname{Cong}(A),
\qquad
\psi(D)=\theta(D_{(x_0)}).
\end{equation}
\end{proposition}

\begin{proof}
For notational convenience, we let $X=\operatorname{Cong}^h_{\operatorname{sat}}(S).$ The order-preserving part is clear. The assignments are well defined: $C^h$ is homogeneous, its
$x_0$-saturation remains homogeneous and is $x_0$-saturated by
construction, while $\theta(D_{(x_0)})$ is a congruence on $A$. We first show that
$C=\psi(\varphi(C))$.

Let $C$ be a congruence on $A$ and $(f,g)\in C$. The homogenization
$(f\sim g)^h$ lies in
\[
C^h\subseteq(C^h:x_0^\infty)=\varphi(C),
\]
and dehomogenizes to $f\sim g$. Hence
$C\subseteq\psi(\varphi(C)).$

Conversely, let a relation in $\psi(\varphi(C))$ be given. Before
applying $\theta$, write it in $(S_{x_0})_0$ as\footnote{We read this in $(S_{x_0})_0$ via \eqref{eq: dehomo}.}
\[
\frac{F}{x_0^d}\sim\frac{G}{x_0^d},
\]
where $F$ and $G$ are homogeneous of degree $d$. By  Theorem~\ref{theorem: Tanaka} (4), there exists
$m\geq0$ such that
\[
x_0^mF\sim_{\varphi(C)}x_0^mG.
\]

Since
$\varphi(C)=(C^h:x_0^\infty),$
there exists $N\geq0$ such that
$x_0^{N+m}F\sim_{C^h}x_0^{N+m}G.$
Let
\[
\delta:S\longrightarrow A,
\qquad
\delta(x_0)=1,\quad \delta(x_i)=t_i.
\]
The congruence $\delta^{-1}(C)$ contains every generating relation of
$C^h$, because the dehomogenization of each homogenized relation in
$C$ is the original relation. Hence
$C^h\subseteq\delta^{-1}(C).$
Applying $\delta$ gives
$\delta(F)\sim_C\delta(G).$
Since
\[
\theta\left(\frac{F}{x_0^d}\right)=\delta(F)
\qquad\text{and}\qquad
\theta\left(\frac{G}{x_0^d}\right)=\delta(G),
\]
the original relation lies in $C$. Therefore
$\psi(\varphi(C))\subseteq C,$
and hence
$\psi(\varphi(C))=C.$

Let $D\in X$, i.e., $D$ is a homogeneous and $x_0$-saturated
congruence on $S$. We show that $\varphi(\psi(D))=D$. Suppose
$f\sim g\in\psi(D)$, and write
\[
p=\deg(f),\qquad q=\deg(g),\qquad m=\max\{p,q\}.
\]
We use the conventions $\deg(0)=0$ and $0^h=0$. In $(S_{x_0})_0$, the two sides are represented by
$f^h/x_0^p$ and $g^h/x_0^q$. By Theorem~\ref{theorem: Tanaka}(4), applied to localization at the
powers of $x_0$, there exists $N\geq0$ such that
\begin{equation}\label{eq: prop722}
x_0^{N+q}f^h\sim_Dx_0^{N+p}g^h.
\end{equation}
With $L=N+p+q-m$,  we can write \eqref{eq: prop722} as\footnote{Here, we use the notation as in \eqref{eq: homoge notation}.}
\[
x_0^L(f\sim g)^h.
\]
Since $D$ is $x_0$-saturated, $(f\sim g)^h\in D$.

Thus every generating relation of $\psi(D)^h$ lies in $D$, and hence
$\psi(D)^h\subseteq D.$
Since $D$ is $x_0$-saturated, it follows that
\[
\varphi(\psi(D))
=
(\psi(D)^h:x_0^\infty)
\subseteq D.
\]

For the reverse inclusion, let $F\sim_DG$ be an arbitrary relation,
and write
\[
F=\sum_dF_d,\qquad G=\sum_dG_d
\]
for their homogeneous decompositions. Since $D$ is homogeneous,
$F_d\sim_DG_d$
for every $d$ by Theorem \ref{theorem: congruence homogeneous elts}. It is therefore enough to prove that every homogeneous
relation $F_d\sim_DG_d$ belongs to $\varphi(\psi(D))$, and then add
these relations degree by degree.

Thus let $F\sim_DG$ be homogeneous, with both sides of degree $d$.
Let $f$ and $g$ be their dehomogenizations, and let
$p=\deg(f),\ q=\deg(g),\ m=\max\{p,q\}.$
Then
\[
F=x_0^{d-p}f^h,\qquad G=x_0^{d-q}g^h.
\]
After localizing at $x_0$, dividing both sides by $x_0^d$, and applying
$\theta$, the relation $F\sim_DG$ gives
\[
f\sim_{\psi(D)}g.
\]
Consequently,
$(f\sim g)^h\in\psi(D)^h
\subseteq\varphi(\psi(D)).$ Since
$(F\sim G)=x_0^{d-m}(f\sim g)^h,$
we obtain
$F\sim G\in\varphi(\psi(D)).$ Applying this argument to every homogeneous component and adding the
resulting relations gives
$D\subseteq\varphi(\psi(D)).$
Therefore $\varphi(\psi(D))=D$.
\end{proof}

\begin{lemma}
\label{lem:prime-separation-saturated}
Let $R$ be an idempotent semiring, let $E$ be a
proper congruence on $R$, and let $u\in R$. Suppose that
$(E:u^\infty)=E.$
If $P\supseteq E$ is a prime congruence with $(u,0)\in P$, then there
exists a prime congruence $Q$ such that
$E\subseteq Q\subsetneq P
\ \text{and}\ (u,0)\not\in Q.$ 
\end{lemma}
\begin{proof}
First observe that the hypotheses imply $(u,0)\notin E$. Indeed, if
$u\sim_E0$, then for every $a,b\in R$ one has
$ua\sim_E0\sim_Eub,$
and therefore
$(a,b)\in(E:u^\infty)=E.$
This would make $E$ the improper congruence, contrary to hypothesis.
In particular, $u\neq0$.

Pass to
$\overline R:=R/E,$
and denote the images of $u$ and $P$ by $\overline u$ and
$\overline P$. The saturation hypothesis says that multiplication by
every power of $\overline u$ is injective on $\overline R$. Indeed, if
$\overline u^{\,m}\overline a = \overline u^{\,m}\overline b,$
then $(u^ma,u^mb)\in E$, and hence
$(a,b)\in(E:u^\infty)=E.$

All products of pairs below are twisted products. Define
\[
\Sigma
:=
\left\{
(\overline u^m,0)\alpha \mid
m\geq0,\ 
\alpha\in(\overline R\times\overline R)\setminus\overline P
\right\}.
\]
We first show that
$\Sigma\cap\Delta_{\overline R}=\varnothing.$
Indeed, if
$(\overline u^m,0)(\alpha_1,\alpha_2)$
is diagonal, then
$\overline u^m\alpha_1 = \overline u^m\alpha_2.$
Injectivity of multiplication by $\overline u^m$ gives
$\alpha_1=\alpha_2$, contrary to
$(\alpha_1,\alpha_2)\notin\overline P$.

The set $\Sigma$ is closed under twisted products. If
$\sigma_j=(\overline u^{m_j},0)\alpha_j\in\Sigma
\ , \ (j=1,2),$
then
$\sigma_1\sigma_2 = (\overline u^{m_1+m_2},0)(\alpha_1\alpha_2).$
Since $\overline P$ is prime and
$\alpha_1,\alpha_2\notin\overline P$, one has
$\alpha_1\alpha_2\notin\overline P$. Hence
$\sigma_1\sigma_2\in\Sigma$.

Consider the collection of congruences on $\overline R$ that are
disjoint from $\Sigma$, ordered by inclusion. This collection is
nonempty because $\Delta_{\overline R}\cap\Sigma=\varnothing$. The
union of any chain in this collection is again a congruence disjoint
from $\Sigma$. Thus, by Zorn's lemma, there exists a congruence
$\overline Q$ maximal among those disjoint from $\Sigma$.

We claim that $\overline Q$ satisfies the congruence-product
primality criterion. For congruences $\Phi_1,\Phi_2$ on $\overline R$,
write
\[
\Phi_1\cdot_{\mathrm{tw}}\Phi_2
:=
\{\alpha\beta \mid \alpha\in\Phi_1,\ \beta\in\Phi_2\}.
\]
Suppose that
$\overline Q\subsetneq\Phi_1, \ \overline Q\subsetneq\Phi_2.$
By maximality of $\overline Q$, both $\Phi_1$ and $\Phi_2$ meet
$\Sigma$. Choose
$\sigma_i\in\Phi_i\cap\Sigma
\ , \ (i=1,2).$
Since $\Sigma$ is closed under twisted products,
$\sigma_1\sigma_2\in\Sigma,$
while
$\sigma_1\sigma_2 \in \Phi_1\cdot_{\mathrm{tw}}\Phi_2.$
Consequently,
$\Phi_1\cdot_{\mathrm{tw}}\Phi_2
\not\subseteq\overline Q.$
Thus, whenever
$\Phi_1\cdot_{\mathrm{tw}}\Phi_2
\subseteq\overline Q$
for congruences $\Phi_1,\Phi_2\supseteq\overline Q$, one has
$\Phi_1=\overline Q$ or $\Phi_2=\overline Q$. By
Remark~\ref{rem:Rowen-primality-comparison} below, this is equivalent,
in the present idempotent setting, to the pairwise twisted-product
primality used throughout this paper. Hence $\overline Q$ is a prime
congruence.

Taking $m=0$ shows that every pair outside $\overline P$ belongs to
$\Sigma$. Since $\overline Q\cap\Sigma=\varnothing$, it follows that
$\overline Q\subseteq\overline P.$
Moreover,
$(\overline u,0)\in\Sigma,$
by taking $m=1$ and
$\alpha=(1,0)\notin\overline P$. Hence
$(\overline u,0)\notin\overline Q.$
Since $(\overline u,0)\in\overline P$, we obtain
$\overline Q\subsetneq\overline P.$
Pulling $\overline Q$ back along $R\to\overline R$ gives the required
prime congruence $Q$.
\end{proof}

\begin{rmk}
\label{rem:Rowen-primality-comparison}
The maximal-disjoint argument in the preceding proof is adapted from
\cite[Lemma~4.12]{Rowen24}. Rowen called a congruence $\Phi$
\emph{prime} when
$\Phi_1\cdot_{\mathrm{tw}}\Phi_2\subseteq\Phi$
for congruences $\Phi_1,\Phi_2\supseteq\Phi$ implies
$\Phi_1=\Phi$ or $\Phi_2=\Phi$, whereas he called the pairwise
condition
\[
\alpha\beta\in\Phi
\Longrightarrow
\alpha\in\Phi\ \text{or}\ \beta\in\Phi
\]
\emph{strong primality}; see
\cite[Definition~4.8(iv)--(v)]{Rowen24}.

To compare these notions here, regard $\overline R$ as the degenerate
pair $(\overline R,\overline R)$ of
\cite[Example~2.9(ii)]{Rowen24}, equipped with the identity negation
map. In Rowen's notation, $a^\dagger$ denotes the image of $a$ under
the negation map, and
$a^\circ:=a+a^\dagger$
is the corresponding quasi-zero. Since the negation map is the
identity in the present case, $a^\dagger=a$. Moreover, since
$\overline R$ is idempotent,
\[
e:=1+1=1,
\qquad
a^\circ=a+a=a,
\qquad
a+a^\circ=a^\circ.
\]
Thus the pair has positive $e$-type, in fact $e$-type $1$. Rowen
states in \cite[Definition~4.8(v)]{Rowen24} that prime and strongly
prime congruences coincide for commutative semiring pairs of positive
$e$-type. Therefore the congruence-product criterion used in the
preceding proof gives precisely the pairwise twisted-product notion of
prime congruence used throughout this paper.
\end{rmk}

\begin{proposition}
\label{prop:saturated-principal-open-dense}
Let $R$ be an idempotent semiring, let $E$ be a
congruence on $R$, and let $u\in R$. If
$(E:u^\infty)=E,$
then
$V(E)\cap D(u)$
is dense in $V(E)$, where
\[
D(u):=\{Q\in\operatorname{SpecCong}(R):(u,0)\notin Q\}.
\]
\end{proposition}

\begin{proof}
The assertion is immediate if $V(E)=\varnothing$, so assume that $E$
is proper. Let $P\in V(E)$. If $(u,0)\notin P$, then
$P\in V(E)\cap D(u)$. If $(u,0)\in P$, Lemma
\ref{lem:prime-separation-saturated} gives a prime congruence $Q$ with
$E\subseteq Q\subsetneq P
\ \text{and}\ 
(u,0)\notin Q.$
Thus $Q\in V(E)\cap D(u)$. Moreover, by definition
\[
\overline{\{Q\}}=V(Q),
\]
and $Q\subseteq P$ implies $P\in V(Q)$. Since
$\{Q\}\subseteq V(E)\cap D(u)$, it follows that
$P\in\overline{V(E)\cap D(u)}. $ Hence every point of $V(E)$
lies in the closure of $V(E)\cap D(u)$.
\end{proof}

We need the following lemma to move from the affine case to the projective case. 

\begin{lemma}
\label{lem:chartwise-saturation}
Let $C$ be a homogeneous $x_0$-saturated congruence on
$S=\mathbb B[x_0,\ldots,x_n]$. For each $i$, set
$R_i:=(S_{x_i})_0, \
E_i:=C_{(x_i)}, \
u_i:=\frac{x_0}{x_i}\in R_i.$
Then $E_i$ is $u_i$-saturated.
\end{lemma}

\begin{proof}
Let $\alpha,\beta\in R_i$. After choosing a common denominator, write
$\alpha=\frac{F}{x_i^d}, \
\beta=\frac{G}{x_i^d},$
with $F$ and $G$ homogeneous of degree $d$. Suppose that
$u_i^N\alpha\sim_{E_i}u_i^N\beta.$
Since
\[
u_i^N\alpha=\frac{x_0^NF}{x_i^{d+N}}
\qquad\text{and}\qquad
u_i^N\beta=\frac{x_0^NG}{x_i^{d+N}},
\]
Theorem~\ref{theorem: Tanaka} (4), applied to localization at the
powers of $x_i$, gives some $m\geq0$ such that
\begin{equation}
x_i^mx_0^NF\sim_Cx_i^mx_0^NG.    
\end{equation}
Since $C$ is $x_0$-saturated, we have
$x_i^mF\sim_Cx_i^mG.$
Localizing at $x_i$ and passing to degree zero gives
$\alpha\sim_{E_i}\beta$. Thus $(E_i:u_i^\infty)=E_i$.
\end{proof}

Now, we prove our main theorem of this section that the projective closure is the topological closure. 

Recall that by the definition, with $S=\mathbb{B}[x_0,\dots,x_n]$ and $A=\mathbb{B}[t_1,\dots,t_n]$, we have
\begin{equation}
  U_0=\operatorname{SpecCong}((S_{x_0})_0)  
\end{equation}
and the
dehomogenization isomorphism
\begin{equation}\label{eq: identification}
(S_{x_0})_0\cong\mathbb B\left[\frac{x_1}{x_0},\ldots,\frac{x_n}{x_0}\right]\cong\mathbb B[t_1,\ldots,t_n], \quad \text{where } t_j=\frac{x_j}{x_0}.
\end{equation}
induces a canonical homeomorphism
\begin{equation}\label{eq: homeo}
\operatorname{SpecCong}(A)\xrightarrow{\sim}U_0\subseteq\Proj.
\end{equation}
Let $C$ be congruence on $A$. Then, we may consider $V(C)$ as a subspace of $\Proj$ through \eqref{eq: homeo}.

\begin{theorem}\label{theorem: topological closure}
\label{thm:projective-congruence-closure}
With the same notation as above, we have
\[
V_+(C^{\mathrm{proj}})
=
\overline{V(C)}^{\,\Proj},
\]
where $\overline{V(C)}^{\,\Proj}$ denotes the topological
closure of $V(C)$ in $\Proj$. In particular, for every ideal $P\subseteq A$, one has
\[
\operatorname{Proj}(P) := V_+\!\left(C_P^{\mathrm{proj}}\right)
=\overline{V(\operatorname{Bend}(P))}^{\,\Proj}.
\]
Thus the projective bend locus $V(\Bend(P))$ is the topological projective closure of the affine bend locus.
\end{theorem}
\begin{proof}
Set $X:=V_+(C^{\mathrm{proj}})$. By Proposition \ref{proposition: closed proposition}, $X$ is closed in $\Proj$. 
By Proposition~\ref{prop:congruence-hom-dehom} and Lemma \ref{lemma: main lemma}, we also have
\[
X\cap U_0=V(C).
\]
It remains to prove that $X\cap U_0$ is dense in $X$. For each standard chart $U_i$, let
\[
R_i=(S_{x_i})_0,
\quad
E_i=(C^{\mathrm{proj}})_{(x_i)},
\quad
u_i=\frac{x_0}{x_i}.
\]
Then, we have
$X\cap U_i=V(E_i),$
and the standard overlap is
$U_i\cap U_0=D(u_i).$
Consequently,
\begin{equation}
(X\cap U_0)\cap U_i=V(E_i)\cap D(u_i).
\end{equation}
The congruence $C^{\mathrm{proj}}$ is $x_0$-saturated by construction,
so Lemma~\ref{lem:chartwise-saturation} shows that $E_i$ is
$u_i$-saturated. Proposition
\ref{prop:saturated-principal-open-dense} therefore shows that
$V(E_i)\cap D(u_i)$ is dense in $V(E_i)$.

Thus $(X\cap U_0)\cap U_i$ is dense in $X\cap U_i$ for every $i$.
Since the standard charts form an open cover of $X$, and
$(X\cap U_0)\cap U_i$ is dense in $X\cap U_i$ for every $i$, the
subset $X\cap U_0$ is dense in $X$. As $X$ is closed, we have
\[
X =
\overline{X\cap U_0}^{\,\Proj} = \overline{V(C)}^{\Proj}.
\]
Taking $C=\operatorname{Bend}(P)$ gives the final assertion.
\end{proof}

\begin{corollary}
\label{cor:projective-irreducibility}
Let $P\subseteq\mathbb B[t_1,\ldots,t_n]$ be an ideal, and suppose
that $V(\operatorname{Bend}(P))\neq\varnothing$. Then the following
are equivalent:
\begin{enumerate}
\item $\operatorname{Proj}(P)$ is irreducible;
\item $V(\operatorname{Bend}(P))$ is irreducible;
\item $\sqrt[\operatorname{pr}]{\operatorname{Bend}(P)}$ is a prime
congruence.
\end{enumerate}
\end{corollary}
\begin{proof}
By Theorem \ref{theorem: topological closure}, we have
\[
\operatorname{Proj}(P)
=
\overline{V(\operatorname{Bend}(P))}^{\Proj}.
\]
A nonempty subspace is irreducible if and only if its closure is
irreducible, so the first two conditions are equivalent. The
equivalence of the second and third conditions follows from
Proposition~\ref{prop:Tanaka-irreducibility}.
\end{proof}

\bibliography{references}\bibliographystyle{alpha}

\end{document}